\RequirePackage{rotating}
\documentclass[12pt]{amsart}
\usepackage[margin=1.5in,footskip=0.25in]{geometry}
\usepackage[graphicx]{realboxes}
\usepackage{float}
\restylefloat{table}
\usepackage{placeins}

\usepackage{caption}
\usepackage{amsmath}
\usepackage{amssymb}
\usepackage{epsfig}
\usepackage{wasysym}
\usepackage{graphicx}

\usepackage{tikz}
\usepackage{tikz-cd}
\usepackage{bm}
\usepackage{xcolor}
\usepackage{listings}
\numberwithin{equation}{section}
\usepackage{extpfeil}
\usepackage{array}
\usepackage{adjustbox}
\usepackage{longtable}
\usepackage{makecell}
\usepackage[all]{xy}
\usepackage{longtable}
\usepackage{rotating}
\input xy
\xyoption{all}
\usepackage{multirow}

\usetikzlibrary{positioning}
\usepackage[colorlinks=false,urlbordercolor=white]{hyperref}
  
\tikzset{sgplattice/.style={inner sep=1pt,norm/.style={red!50!blue},char/.style={blue!50!black},
  lin/.style={black!50}},cnj/.style={black!50,yshift=-2.5pt,left=-1pt of #1,scale=0.5,fill=white}}

\DeclareFontFamily{U}{mathb}{\hyphenchar\font45}
\DeclareFontShape{U}{mathb}{m}{n}{
      <5> <6> <7> <8> <9> <10> gen * mathb
      <10.95> mathb10 <12> <14.4> <17.28> <20.74> <24.88> mathb12
      }{}
\DeclareSymbolFont{mathb}{U}{mathb}{m}{n}
\DeclareMathSymbol{\righttoleftarrow}{3}{mathb}{"FD}

\calclayout
\allowdisplaybreaks[3]

\theoremstyle{plain}
\newtheorem{prop}{Proposition}[section]

\theoremstyle{definition}

\newtheorem{rema}[prop]{Remark}

\newtheorem{exam}[prop]{Example}

\newcommand{\actsfromright}{\righttoleftarrow}

\newcolumntype{C}[1]{>{\centering\let\newline\\\arraybackslash\hspace{0pt}}m{#1}}

\def\cD{{\mathcal D}}

\def\cO{{\mathcal O}}

\def\rg{{\mathrm g}}

\def\Am{{\mathrm{Am}}}

\def\sA{{\mathsf A}}

\def\bA{{\mathbb A}}

\def\bP{{\mathbb P}}

\def\bZ{{\mathbb Z}}

\def\bC{{\mathbb C}}

\def\rH{{\mathrm H}}

\def\Br{\mathrm{Br}}

\def\Bl{\mathrm{Bl}}
\def\Pic{\mathrm{Pic}}

\def\Gal{\mathrm{Gal}}
\def\Aut{\mathrm{Aut}}

\def\Hom{\mathrm{Hom}}

\def\lim{\mathrm{lim}}

\def\Spec{\mathrm{Spec}}

\def\Cr{\mathrm{Cr}}

\def\Cl{\mathrm{Cl}}

\def\Z{\mathbb Z}

\makeatother
\makeatletter

\begin{document}
\title[]{Computing the equivariant Brauer group}
\author[]{Alena Pirutka and Zhijia Zhang}

\address{Courant Institute of Mathematical Sciences, New York University, New York, 10012, U.S.A.}
\email{pirutka@cims.nyu.edu}

\email{zhijia.zhang@cims.nyu.edu}

\date{\today}

\begin{abstract}
Let $X$ be a smooth projective rational variety carrying a regular action of a finite abelian group $G$.
We give examples of effective computation of the Brauer group of the quotient stack $[X/G]$ in dimensions $2$ and $3$ using residues in Galois cohomology and the geometry of fixed loci. In particular, we compute $\Br([X/G])$ for all $G$-minimal del Pezzo surfaces.
\end{abstract}

\maketitle

\section{Introduction}

Let $k$ be a field of characteristic $0$. Consider a smooth projective rational variety $X$ over an algebraic closure $\bar k$ of $k$ carrying a regular and generically free action of a finite group $G$. Studying such actions up to equivariant birationality is a classical and active area in birational geometry. Of particular interest is the {\em linearizability problem},  which asks whether or not the $G$-action on $X$ is {\it linearizable}, i.e., equivariantly birational to a linear $G$-action on $\bP^n.$ An arithmetic counterpart of this problem is the classical {\em rationality problem} over nonclosed fields $k$, where the Galois action is considered as an analogue of the $G$-action.

An established strategy to study both linearizability and rationality problems is to seek nontrivial birational invariants. The similarities between the two problems are well reflected in the following invariant: the group cohomology \begin{align}\label{eq:H1P}
    \rH^1(H, \mathrm{Pic}(X_{\bar k})),\quad H\subseteq G.
\end{align} 
This invariant was first studied by Yu. Manin \cite{M} in the arithmetic setting, where $G=\Gal(\bar k/k)$ is the Galois group and the action is induced from the Galois action on $\bar k$.
F. Bogomolov and Y. Prokhorov \cite{BP} extended it to the equivariant setting, when $G$ is a finite group and the action comes from geometric automorphisms of $X_{\bar k}$. In the respective cases, the vanishing of \eqref{eq:H1P} for every subgroup $H$ of $G$ is a necessary condition for $X$ to be stably rational over $k$, and for the $G$-action on $X_{\bar k}$ to be stably linearizable. Applications of this invariant to the study  of the linearizability problem can be found in \cite{CTZ, HTZ}. 

In the arithmetic setting, when $k$ is a nonclosed field, the Leray spectral sequence for the Galois action gives rise to a well-known long exact sequence
\begin{multline}\label{arithss}
0\to\Pic(X)\to\Pic(X_{\bar k})^{\Gal(\bar k/k)}\to\Br(k)\to\\
\to\ker(\Br(X)\to\Br(X_{\bar k}))\to\rH^1(\Gal(\bar k/k),\Pic(X_{\bar k}))\to\rH^3(\Gal(\bar k/k),\bar k^\times),
\end{multline}
where $\Br(k)$ and $\Br(X)$ are the Brauer group of $k$ and $X$ respectively. The group $\ker(\Br(X)\to\Br(X_{\bar k}))$ is known as the {\em algebraic part} of the Brauer group. 

In the equivariant setting, when $k=\bar k$, the Leray spectral sequence for the $G$-action 
produces an exact sequence similar to \eqref{arithss}
\begin{multline}\label{ss}
0\to\Hom(G,k^\times)\to\Pic(X,G)\to\Pic(X)^G\to\rH^2(G,k^\times)\to\\\to\Br([X/G])\to\rH^1(G,\Pic(X))\to\rH^3(G,k^\times),
\end{multline}
where $\Pic(X,G)$ is the group of $G$-linearizable line bundles on $X$ and $\Br([X/G])$ is the Brauer group of the quotient stack $[X/G]$. The group $\Br([X/G])$ is a $G$-stably birational invariant, and can be viewed as an analogue of the algebraic part of the Brauer group as in \eqref{arithss}.

From now on, we focus on the equivariant setting with $k=\bar k$, and study the groups $\Br([X/G])$ and $\rH^1(G,\Pic(X))$. Given a $G$-action on $X$, it can be computationally challenging to find the induced $G$-action on $\mathrm{Pic}(X)$, as this requires a thorough analysis of divisors on $X$. On the other hand, the geometry of the fixed locus $X^G$ contains rich information readily available in the equivariant setting, but absent in the arithmetic setting (where the Galois fixed locus simply consists of all $k$-rational points). 

When $G$ is a cyclic group acting on a smooth rational surface $X$ with {\em maximal} stabilizers, the works of F. Bogomolov and Y. Prokhorov, and E. Shinder \cite{BP,S} give a formula for $\rH^1(G,\Pic(X))$ only involving information about the $G$-fixed curves on $X$. Generalizing this, for any finite group $G$ acting on a smooth projective variety $X$, A. Kresch and Y. Tschinkel gave an algorithm to compute $\Br([X/G])$ which only requires information about divisors with nontrivial stabilizers, and presented
examples of effective computations when $X$ is a rational surface \cite{KT22,KT24}.

In this note, we extend the scope of applications of this algorithm to dimension $3$. In particular, we produce a nontrivial class in $\Br([\tilde X/G])$ for a smooth model $\tilde X$ of a singular cubic threefold $X$, and showcase the connection between $\Br([\tilde X/G])$ and $\rH^1(G,\Pic(\tilde X))$ through this example (see Remark~\ref{rema:connection}). We also complete the computations of $\Br([X/G])$ in dimension 2 for all $G$-minimal del Pezzo surfaces $X$ and finite abelian groups $G$.

Here is a road map of the paper. In Section~\ref{reminders}, we review basic facts about the Brauer group of the quotient stack. In Section \ref{dim3}, we produce an example with nontrivial $\Br([X/G])$ in dimension 3. We compute $\Br([X/G])$ in dimension 2 in Section \ref{dim2}.

\

\paragraph{\bf Acknowledgements.} We would like to thank Yuri Tschinkel for many helpful discussions, and Andrew Kresch for comments on the manuscript.  The first author was partially supported by NSF grant DMS-2201195.

\section{Preliminaries} \label{reminders}

\subsection{Brauer groups}

Let $X$ be a smooth projective variety over a field $k$ and $n$ a positive integer invertible in $k$. Let $\rH^i(k(X), \mu_n)$ be the Galois cohomology of the function field of $X$ with $\mu_n$-coefficients, where $\mu_n$ is the \'etale $k$-group scheme of the $n^{\mathrm{th}}$-roots of unity (in particular, $\mu_n\simeq\bZ/n\bZ$ when $k$ is algebraically closed).

If $v$ is a discrete valuation on the field $k(X)$, one has the residue maps
\begin{equation*}
\partial_v^i: \rH^i(k(X), \mu_{n}^{\otimes j})\to \rH^{i-1}(\kappa(v), \mu_{n}^{\otimes (j-1)}),
\end{equation*}
where $\kappa(v)$ is the residue field.
In particular, for $a\in \rH^1(k(X),\bZ/2\bZ)$ and for $(a,b)\in \rH^2(k(X),\bZ/2\bZ)$, one has 
\begin{align}\label{eqn:valeq}
    \partial_v^1(a)&=v(a)\mbox{ mod }2 \in  \bZ/2\bZ,\notag\\
    \partial_v^2(a,b)&=(-1)^{v(a)v(b)}\overline{a^{v(b)}b^{-v(a)}}\in\kappa(v)^{\times}/(\kappa(v)^{\times})^2,
\end{align}
where for a unit $u$ in the valuation ring of $v$, we denote by $\bar u$ its image in the residue field $\kappa(v)$.

For an irreducible divisor $D$ on $X$, we denote by $v_D$ the associated divisorial valuation on $k(X)$, $\partial_D^i$  the corresponding residue maps, and $\kappa(v_{D})$ or $\kappa(D)$ the residue field. Similarly, for $\xi\in X^{(1)}$ a codimension $1$ point of $X$, we denote by $v_\xi$ and $\partial_\xi^i$ the associated valuation and residue maps.
The $n$-torsion in the Brauer group of $X$ can be computed via

\begin{equation}\label{Brclassical}
\Br(X)[n]=\bigcap_{D} \mathrm{ker}(\partial^2_D),
\end{equation}
where $D$ runs over all irreducible divisors on $X$ (see \cite[Proposition 4.2.3]{CT} and \cite[Theorem 4.1.1]{CT}; note that we assume that $X$ is smooth and projective).

\medskip

Now let $k$ be an algebraically closed field of characteristic zero and $X$ a smooth projective variety over $k$ carrying a generically free regular action of a finite group $G$. An analogue of the classical formula \eqref{Brclassical} for $\Br([X/G])$, the Brauer group of the quotient stack $[X/G]$, is established in \cite{KT22, KT24}:

\begin{prop}\label{remres}
Let $k$ be an algebraically closed field of characteristic zero and $X$ a smooth  projective variety over $k$ carrying a generically free regular action of a finite group $G$.
For any irreducible divisor $D$ on $X$, we denote by $I_D$ the stabilizer group
$$
I_D:=\{g\in G \,|\, g\mbox{ acts trivially on }D\}.
$$
Then the $n$-torsion subgroup of $\Br([X/G])$, denoted by $\Br([X/G])[n]$, can be computed via
\begin{equation}\label{Brstack}
\Br([X/G])[n]=\bigcap_{D} \mathrm{ker}\left( |I_D|\cdot \partial^2_{D'}\right)\subset H^2(k(X)^G, \mathbb Z/n\bZ),
\end{equation}
where $D$ runs over all irreducible divisors on $X$, $|I_D|$ is the order of $I_D$, 
and $\partial^2_{D'}$ is the  residue map in degree $2$ corresponding to the divisorial valuation on $k(X)^G$ given by the image $D'$ of $D$.
\end{prop}

\begin{proof}
See \cite[Proposition 4.2]{KT22}, \cite[Section 4]{KT24}, and \cite[Proposition 2.2]{KT19}. 
\end{proof}

\subsection{Rational surfaces}
In this subsection, we review an effective algorithm provided in \cite{KT22} to compute $\Br([X/G])$   when $X$ is a rational surface. 

Let $X$ be a smooth projective rational surface carrying a generically free regular action of a finite group $G.$ Recall that the $G$-action on $X$ is called {\em in standard form } if there exists a $G$-invariant simple normal crossing divisor $\cD\subset X$ such that
\begin{itemize}
    \item the $G$-action on $X\setminus \cD$ is free, and
    \item for any $g\in G$ and any irreducible component $D$ of $\cD$, either $g(D)=D$ or $g(D)\cap D=\emptyset$.   
\end{itemize}
Any $G$-action on $X$ can be brought into standard form via successive blowups \cite[Theorem 3.2]{RYessential}. The following proposition provides an effective algorithm only involving divisors with nontrivial stabilizers to determine $\Br([X/G])$.

\begin{prop}[{\cite[Proposition 4.2, Corollary 4.6]{KT22}}]\label{prop:ratsurf}
    Let $X$ be a smooth projective rational surface carrying a finite group $G$-action in standard form. Then the group $\Br([X/G])[n]$ can be identified with the kernel of the map 
    \begin{align}\label{eqn:surfaceseq}
          \bigoplus_{[\xi]\in X^{(1)}/G}\rH^1(\Spec(k(\xi)^{D_\xi}),\bZ/n\bZ)\stackrel{\oplus \partial^1}{\longrightarrow}\bigoplus_{[\mathfrak p]\in X/G}\bZ/n\bZ,       
    \end{align}
    where the sum on the left runs over $G$-orbit representatives $[\xi]$ of codimension 1 points on $X$ such that the stabilizer group $I_\xi$ has cardinality $n$,
    $$
    D_\xi:=\{g\in G\mid \xi\cdot g=\xi\},
    $$
    and the sum on the right runs over $G$-orbit representatives $[\mathfrak p]$ of points of $X$.
\end{prop}
 Using Proposition~\ref{prop:ratsurf}, we compute $\Br([X/G])$ for all $G$-minimal del Pezzo surfaces with abelian groups $G$ in Section~\ref{dim2}. Here we introduce the notation and demonstrate the process in detail with a concrete example. 
\begin{exam}\label{exam:cubexam}
     We consider the case $\fbox{3.36}$ in \cite{B06}. Let $X$ be a smooth cubic surface given by 
     $$
     \{w^3+x^3+xy^2+z^3=0\}\subset\bP^3_{w,x,y,z},
     $$
         with an action of $G=C_3\times C_6$ generated by 
         \begin{align*}
             \sigma\colon&  (w,x,y,z)\mapsto(\zeta_3 w,x,y,z),\\
            \tau\colon& (w,x,y,z)\mapsto (w,x,-y,\zeta_3 z).
         \end{align*}
  We list the stratification of curves with nontrivial stabilizers
$$
    \centering
    \begin{tabular}{c|c|c|c|c|c}
    $i$&Curve $\xi_i$&$I_{\xi_i}$&$D_{\xi_i}$&$\rg(\xi_i)$&$\rg(\xi_i/D_{\xi_i})$\\\hline
    1&$\{y=0\}\cap X $ &  $C_2=\langle\tau^3\rangle$ &$G$ & 1&0\\\hline
    2&$ \{w=0\}\cap X$& $C_3=\langle\sigma\rangle$&$G$& 1&0 \\\hline
      3&$\{z=0\}\cap X$& $C_3=\langle\tau^2\rangle$&$G$&1&0
    \end{tabular}
$$
where 
\begin{itemize}
    \item the second column displays equations of the curves $\xi_i$ with nontrivial generic stabilizers;
    \item the third column displays the stabilizers $I_{\xi_i}$ of $\xi_i$;
    \item the fourth column displays the groups $D_{\xi_i}=\{g\in G\mid \xi_i\cdot g=\xi_i\}$;
    \item the fifth column displays the genera of $\xi_i$;
    \item the sixth column displays the genera of $\xi_i/D_{\xi_i}$, computed via the Riemann-Hurwitz formula.
    \end{itemize}
    
 One can check that the $G$-action on $X$ is in standard form: the $G$-action is free in the complement of the simple normal crossing divisor $\xi_1\cup\xi_2\cup\xi_3$ in $X$, and $G$ leaves invariant each of the curves $\xi_i$ for $i=1,2,3.$ By Proposition \ref{prop:ratsurf}, we know that $\Br([X/G])[n]$ is possibly nonzero only when $n=2$ or $3$.

For $n=3,$ observe that $\xi_2\cap\xi_3$ consists of three points where two of them are in the same $G$-orbit. It follows that the quotient curves $\xi_2/G$ and $\xi_3/G$ are both rational and meet at two points, denoted by $p_1$ and $p_2$. The kernel of \eqref{eqn:surfaceseq} consists of classes of functions ramified only at $p_1$ and $p_2$, where the ramification indices are distinct at two points. It follows that 
$$
\Br([X/G])[3]=\bZ/3.
$$ 

For $n=2,$ the quotient curve $\xi_1/G$ has genus $0.$ Then 
$$
\ker\left(\rH^1(\Spec(k(\xi_1))^G,\bZ/2\bZ)\to\bigoplus_{[\mathfrak p]\in X/G}\bZ/2\bZ\right)=0
$$
(see \cite[Proposition 4.2.1(b)]{CT}).
Combining the $2$-torsion and $3$-torsion subgroups, we conclude that
$$
\Br([X/G])=\bZ/3\bZ.
$$
\end{exam}

\section{A cubic threefold} \label{dim3}
Let $k$ be an algebraically closed field of characteristic zero, and $X$ the cubic threefold given by 
\begin{align*}
   X=\{F=0\}\subset\bP^4_{y_1,\ldots,y_5},
\end{align*}
where 
\begin{align*}
    F&=y_3^2(y_1-y_4)+y_5^2(y_1+y_2)-2y_1y_3y_5+f,\\
    f&=-y_1^2y_2-y_1y_2^2+y_1^2y_4-y_1y_4^2+y_2^2y_4-y_2y_4^2-2y_1y_2y_4.
\end{align*}
Let $G=C_2$ act on $X$  by 
$$
(y_1,y_2,y_3,y_4,y_5)\mapsto (y_1,y_2,-y_3,y_4,-y_5).
$$
The singular locus of $X$ consists of six ordinary double points:
$$
 p_1=[0 : 0 : 0 : 1 : -1],\quad p_2=[0 : 0 : 0 : 1 : 1],\quad p_3=[1 : 0 : -1 : 0 : -1],
 $$
 $$
 p_4=[1 : 0 : 1 :
0 : 1],\quad p_5=[0 : 1 : -1 : 0 : 0],\quad p_6=[0 : 1 : 1 : 0 : 0].
$$
Let $\tilde X$ be the blowup of $X$ at $p_1,\ldots,p_6$ and the $G$-invariant curve given by 
$$
\{y_3=y_5=0\}\cap X.
$$ 
Since $p_1,\ldots,p_6$ are ordinary double points, $\tilde X$ is smooth.
\begin{prop}\label{exd3}
In the above notation, the group
$
\Br([\tilde X/G])
$
is nonzero.
More precisely, the class
$$\alpha=\left(\frac{f}{(y_2-y_4)^3}, \frac{-y_1y_4+y_1y_2-y_2y_4}{(y_2-y_4)^2}\right)\in\rH^2(k(X)^G,\bZ/2\bZ)$$
is nontrivial, and belongs to $\Br([\tilde X/G]).$
In particular, the $G$-action on $X$ is not linearizable.

\end{prop}

\bigskip

The rest of this section is devoted to the proof of Proposition \ref{exd3}.  
First, let 
$$
q=-y_1y_4+y_1y_2-y_2y_4
$$ and
\begin{align}\label{eqn:alpha6a1}
    \alpha=(a,b)\in\rH^2(k(X)^G,\bZ/2\bZ),
\end{align}
where 
$$
 a:=\frac{f}{(y_2-y_4)^3}, \quad b:=\frac{q}{(y_2-y_4)^2}.
$$

Since $\tilde X$ is smooth and projective, by Proposition \ref{remres}, it suffices to check that for all divisorial valuations on $k(\tilde X)^G$ given by the image $D'$ of a divisor $D$ on $\tilde X$, we have
\begin{align}\label{condnr}
|I_{D}|\cdot\partial^2_{D'}(\alpha)=0.
\end{align}

Note that no singular point of $X$ lies on the plane section of $X$ given by $y_3=y_5=0$. In particular, we may blow up the singular points and the curve $\{y_3=y_5=0\}\cap X$ independently.

\subsection*{Blowup of $y_3=y_5=0$}

The blowup $Y$ of $X$ along the cubic curve 
$$
\{y_3=y_5=0\}\cap X
$$ 
is given by 
$$
Y=\{F=y_3z_5-y_5z_3=0\}\subset\bP^4_{y_1,\ldots,y_5}\times\bP^1_{z_3,z_5}.
$$
We first compute residues in the affine chart $U$ of $Y$ given by $y_4=z_3=1.$
Put
$$
g=y_1-y_4+y_1z_5^2+y_2z_5^2-2y_1z_5.
$$ 
Then $y_5=y_3z_5$ and the equation of $U$ is 
$$
U: y_3^2\bar g +\bar f=0,
$$
where
$\bar g$ (resp. $\bar f$) is the affine equation of $g$ (resp. $f$) in the chart $y_4=1$.
The action of $G$ on $U$ is
$$
(y_1, y_2, y_3, z_5)\mapsto (y_1, y_2, -y_3, z_5).
$$
Then $U/G$ is the affine rational variety  
\begin{equation}\label{eqUG}
U/G\subset \mathbb A^4_{y_1, y_2, w_3, z_5}, \quad w_3\bar g +\bar f=0.
\end{equation}
In $k(U)^G$, one has $\bar{f}=-w_3\bar{g}$, and we rewrite
$$
\alpha=\left(\frac{-w_3\bar g}{(y_2-1)^3}, \frac{-y_1+y_1y_2-y_2}{(y_2-1)^2}\right)\in\rH^2(k(U)^G,\bZ/2)=\rH^2(k(X)^G,\bZ/2).
$$

We then compute residues of $\alpha$ at divisors of $U/G\simeq \mathbb A^3_{y_1, y_2, z_5}.$ From the definition of $\alpha$, we know that $\partial_{D'}(\alpha)=0$ for any divisor $D'$ of $U/G$ except possibly when $D'$ is one of the following four divisors:
$$D'_1: w_3=0, \quad D'_2: \bar g=0, \quad D'_3: y_2-1=0, \quad D'_4: -y_1+y_1y_2-y_2=0.$$
We consider these four cases:
\begin{enumerate}
\item $D'_1$ is the image of the divisor $D_1: y_3=0$ on $U$, with $|I_{D_1}|=2$. Hence condition (\ref{condnr}) is satisfied for $D_1$.
We claim that
\begin{equation}\label{nzero}
\partial_{D'_1}(\alpha)=q=-y_1+y_1y_2-y_2\neq 0\text{ in } \kappa(D'_1)^{\times}/(\kappa(D'_1)^{\times })^2.
\end{equation}
Indeed, at the generic point of $D'_1$, the function $\bar{g}$ is invertible. Hence $\kappa(D'_1)$ is the field of functions of the subscheme 
$$
\{\bar f=0\}\subset \mathbb A^3_{y_1, y_2, z_5},
$$
which is a purely transcendental extension of the function field of the cubic curve $$C: \bar f = -y_1^2y_2-y_1y_2^2+y_1^2-y_1+y_2^2-y_2-2y_1y_2=0\subset \mathbb A^2_{y_1, y_2}.$$
Since the function $y_2-1$ is invertible at the generic point of $C$, we may write:
$$
\bar f(y_1, y_2)=\bar f\left(\frac{q+y_2}{y_2-1}, y_2\right)=\frac{y_2^2(-q-5)-4qy_2-q^2-q}{y_2-1}.
$$ 
The discriminant of 
$$
y_2^2(-q-5)-4qy_2-q^2-q
$$ 
as a quadratic polynomial in the variable $y_2$ over $k(q)$ is
$$
d=q(-4q^2-8q-20)
$$ 
with $-4q^2-8q-20$ being a nonsquare in $k(q)$.
We obtain:
$$\kappa(D_1)=k(q)(\sqrt{d})(z_5).$$
Hence the kernel of the natural map $$k(q)^{\times}/(k(q)^{\times})^2\to \kappa(D'_1)^{\times}/(\kappa(D'_1)^{\times})^2$$ is generated by $d$, so that $q$ is a nonzero element in $\kappa(D'_1)^{\times}/(\kappa(D'_1)^{\times})^2.$

\item $\partial_{D'_2}(\alpha)$ is the image of $-y_1+y_1y_2-y_2$ in $\kappa(D'_2)^{\times}/(\kappa(D'_2)^{\times})^2.$
By the definition of $D'_2$, we have a relation 
$$y_1-1+y_1z_5^2+y_2z_5^2-2y_1z_5=0$$ in $\kappa(D'_2)$, where we still write $y_1, y_2, z_5$ for their images in $\kappa(D'_2)^{\times}$. We rewrite this condition as: 
$$(y_1+y_2)(z_5-\frac{y_1}{y_1+y_2})^2+\frac{-y_1+y_1y_2-y_2}{y_1+y_2}=0,$$ so that  $$-y_1+y_1y_2-y_2=-(y_1+y_2)^2(z_5-\frac{y_1}{y_1+y_2})^2$$ is a square in $\kappa(D'_2)^{\times}$ and $\partial_{D'_2}(\alpha)=0.$
\item $\partial_{D'_3}(\alpha)=(-y_1+y_1y_2-y_2)^3|_{y_2=1}=-1$ is a square in $\kappa(D'_3)^{\times}$. So we have that $\partial_{D'_3}(\alpha)=0.$
\item $\partial_{D'_4}(\alpha)$ is the image of $\frac{\bar f}{(y_2-1)^3}$ in $\kappa(D'_4)^{\times}/(\kappa(D'_4)^{\times })^2.$ In the field $\kappa(D'_4)$, 
we have a relation 
$$
y_1=\frac{y_2}{y_2-1}.
$$ 
Then we find that in $\kappa(D'_4)$,
$$\frac{\bar f}{(y_2-1)^3}=-\frac{5y_2^2}{(y_2-1)^4}$$ is a square, hence $\partial_{D'_4}(\alpha)=0.$
\end{enumerate}

The computations in the remaining charts $y_1=z_3=1,$ $y_2=z_3=1,$ and  $y_i=z_5=1, i=1,2,4$ of $Y$ are similar. 
Hence we have verified that the condition \eqref{condnr} holds for all divisors on $\tilde X/G$ except the images of the exceptional divisors of the blowups of six singular points.

\subsection*{Exceptional divisors of $\Bl_{p_1,p_2}(X)$} Let $v$ be the valuation on $k(X)^G$ corresponding to the exceptional divisor of the blowup of $X$ at the $G$-orbit of two singular points $p_1$ and $p_2$.
We work with the affine chart $y_4=1$. Put 
$$
g_1=y_1,\quad g_2=y_2,\quad g_3=y_3,\quad g_4=y_5^2-1, 
$$
one may view the union of $p_1$ and $p_2$ as a variety given by 
$$
\{g_1=g_2=g_3=g_4=0\}\subset\bA^4_{y_1,y_2,y_3,y_5}.
$$
The blowup $\Bl_{p_1,p_2}(\bA^4)$ is given by 
$$
\{g_iz_j-g_jz_i=0 \ \vert\ i,j=1,\ldots,4\}\subset\bA^4_{y_1,y_2,y_3,y_5}\times\bP^3_{z_1,z_2,z_3,z_4}.
$$
The induced $G$-action is given by 
$$
(y_1,y_2,y_3,y_5)\times (z_1,z_2,z_3,z_4)\mapsto(y_1,y_2,-y_3,-y_5)\times (z_1,z_2,-z_3,z_4).
$$
In the affine chart $z_4=1$, the defining equations are equivalent to the change of variables 
\begin{align}\label{eqn:changevarp1p2}
    y_i=z_i(y_5^2-1),\quad i=1,2,3.
\end{align}
The exceptional divisor $E$ is given by $y_5^2=1$. Note that it consists of two components in the same $G$-orbit. So it gives rise to a divisorial valuation $v$ of $k(X)^G$. Since $a,b\in k(X)^G$, after substituting \eqref{eqn:changevarp1p2} into \eqref{eqn:alpha6a1}, one can compute
\begin{align}
    \label{eqn:nualphap1p2}
    v\left(\frac{f}{(y_2-y_4)^3}\right)=v\left(\frac{q}{(y_2-y_4)^2}\right)=1.
\end{align}
From \eqref{eqn:valeq}, one has
\begin{align*}
\partial_v(\alpha)=\frac{f}{q(z_2(y_5^2-1)-1)}=-1\in \kappa(v)^{\times}/(\kappa(v)^{\times})^2
\end{align*}
where the second equality is obtained via evaluation at $y_5^2=1$. 
It follows from \eqref{eqn:nualphap1p2} that $\partial_v(a)=0.$

The computations of the residue of $\alpha$ along exceptional divisors of blowups of the other two $G$-orbits of singular points are similar. We summarize them below. 

\subsection*{Exceptional divisors of $\Bl_{p_3,p_4}(X)$}
Let $v$ be the valuation on $k(X)^G$ corresponding to the exceptional divisors of $\Bl_{p_3,p_4}(X)$. Similarly as \eqref{eqn:changevarp1p2}, after choosing an appropriate affine chart and introducing new coordinates $z_2, z_3, z_4$, the blowup of $\bP^4$ at $p_3$ and $p_4$ can be considered as the change of variables 
\begin{align*}
    y_1&=1,\quad 
&y_2=z_2(y_3^2-1),\\
y_4&=z_4(y_3^2-1),\quad
&y_5=-z_3(y_3^2-1)+y_3.
\end{align*}
Plugging this into \eqref{eqn:alpha6a1}, one can compute 
$$
v\left(\frac{f}{(y_2-y_4)^3}\right)=-2,\qquad v\left(\frac{q}{(y_2-y_4)^2}\right)=-1.
$$
It follows from \eqref{eqn:alpha6a1} that 
$$
\partial_v(\alpha)=-(z_2-z_4)^2
$$
is trivial in $\kappa(v)^{\times}/(\kappa(v)^{\times})^2$.

\subsection*{Exceptional divisors of $\Bl_{p_5,p_6}(X)$}
Let $v$ be the valuation on $k(X)^G$ corresponding to the exceptional divisors of $\Bl_{p_5,p_6}(X)$. Similarly as before, after choosing an appropriate affine chart and introducing new coordinates $z_1, z_4, z_5$, the blowup of $\bP^4$ at $p_5$ and $p_6$ can be considered as the change of variables 
\begin{align*}
   y_2=1,\qquad
y_i=z_i(y_3^2-1),\quad i=1,4,5.
\end{align*}
Plugging this into \eqref{eqn:alpha6a1}, one can compute 
$$
v\left(\frac{f}{(y_2-y_4)^3}\right)=v\left(\frac{q}{(y_2-y_4)^2}\right)=1.
$$
Then $
\partial_v(\alpha)
$
is obtained by evaluating
$$\frac{f}{q(1-z_4(y_3^2-1))}
$$
at $y^3-1=0$. After the above change of variables, this gives $
\partial_v(\alpha)=-1,
$
and thus we know $\partial_v(\alpha)=0$. 

\bigskip

In summary, condition \eqref{condnr}  is satisfied for all divisors on $\tilde X/G$, hence $\alpha$ defines an element of $\Br([\tilde X/G]).$ Moreover, it is nonzero since its residue at $w_3=0$ is nonzero by \eqref{nzero}. Since $G$ is a cyclic group, one has 
$$
\rH^2(G, k^{\times})=\rH^3(G, k^{\times})=0.
$$
The sequence \eqref{ss} implies that 
\begin{align}\label{eq:H1BR6a1}
    \rH^1(G, \Pic(\tilde X))=\Br([\tilde X/G])
\end{align}
so that the $G$-action on $X$ is not (stably) linearizable.

\begin{rema}\label{rema:connection}
For the $G$-action on $X$ given in Proposition~\ref{exd3}, it is also computed in \cite{CTZ} that 
\begin{align}\label{nontrivial}
    \rH^1(G,\Pic(\hat X))=\rH^1(G,\Cl(X))=\bZ/2\bZ
\end{align}
where $\hat X$ is the blowup of $X$ at $p_1,\ldots,p_6.$ 
In particular, the divisor class group $\Cl(X)$ of $X$ is generated by the class $F$ of a general hyperplane section on $X$ and two classes of rational normal cubic scrolls $S_1$ and $S_2$ in $X$ subject to the relation $S_1+S_2=2F$. The $G$-action switches $S_1$ and $S_2$, contributing to nontrivial cohomology \eqref{nontrivial}. 

Our computation above illustrates how to find nontrivial elements in the group $\Br([\tilde X/G])$ via residues in Galois cohomology, without using information of $\Pic(\hat X)$ as in \cite{CTZ} or the group-theoretic formulas as in \cite{KT24}. 

On the other hand, our computation reflects a striking similarity with the computation in \cite{CTZ}, making the equality \eqref{eq:H1BR6a1} explicit. Indeed, with the notation in Proposition~\ref{exd3}, the factor $q=-y_1y_4+y_1y_2-y_2y_4$ in $\alpha$ is a quadric section (equivalent to $2F$ in $\Cl(X)$) cutting out two cubic scrolls on $X$: 
$$
X\cap\{q=0\}=R_1+R_2,
$$
where 
\begin{align*}
   R_1=& \{q=y_2y_5+y_1(\sqrt{5}y_2-y_3+y_5)=y_3y_4-y_1(y_3-\sqrt{5}y_4-y_5)=0\},\\
   R_2=& \{q=y_2y_5-y_1(\sqrt{5}y_2+y_3-y_5)=y_3y_4-y_1(y_3+\sqrt{5}y_4-y_5)=0\}.
\end{align*}
One sees that $R_1$ and $R_2$ are two rational normal cubic scrolls, corresponding to the two classes $S_1$ and $S_2$ in $\Cl(X)$ which contribute to  \eqref{nontrivial}. 
\end{rema}

\section{Rational surfaces} \label{dim2}

Throughout this section, we work over $k=\bC$. Let $\Cr_2(\bC)$ be the plane Cremona group, i.e., the group of birational automorphisms over $\bC$ of $\bP^2$. Finite abelian subgroups of $\Cr_2(\bC)$ have been classified in \cite{B06}. We recall the basic settings. Let $G\subset \Cr_2(\bC)$ be a finite group. It is known that we can find a smooth projective surface $X$ with a regular {\em $G$-minimal} action on $X$ inducing the embedding $G\subset\Cr_2(\bC)$. Here a $G$-minimal action means one of the following two cases holds
\begin{enumerate}
    \item  $\Pic(X)^G=\bZ$ and $X$ is a del Pezzo surface;
     \item $\Pic(X)^G=\bZ^2$ and $X$ is a $G$-conic bundle.
\end{enumerate}
In the remainder of this section, we compute $\Br([X/G])$ in case (1), i.e., for $G$-minimal del Pezzo surfaces $X$ using Proposition~\ref{prop:ratsurf}. \ We focus on the cases when $G$ is a finite abelian group, and rely on a classification of such models, in particular the lists of groups and {\em regular} actions in \cite[Chapter 10]{B06}. We omit technical details of the computation and refer readers to Example~\ref{exam:cubexam} for an illustration of the computation process.  We keep the labels and notation as in {\it loc. cit.} In particular, 
\begin{itemize}
    \item $L_d(x_1,\ldots,x_n)$ denotes a general homogeneous form of degree $d$ in variables $x_1,\ldots, x_n$;
    \item $\zeta_n$ is a primitive $n$-th root of unity;
    \item $\lambda,\mu$ are general complex numbers.
\end{itemize}
\subsection{Cyclic groups}
 Note that for an action of a cyclic group $G$ on a smooth projective variety $X$, we have $\rH^2(G,\bC^\times)=0$. Then \eqref{ss} implies that 
$$
\Br([X/G])=\rH^1(G,\Pic(X)).
$$
Let $G\subset\Cr_2(\bC)$ be a cyclic group generated by an element $\sigma$ of order $n$. 
By \cite[Chapter 10.1]{B06}, we know that up to conjugation in $\Cr_2(\bC)$, the embedding $G\subset\Cr_2(\bC)$ is induced by a $G$-action on $X$ in one of the following cases:

\

\noindent$\bullet$ {\bf Linear automorphisms.}

\noindent$\fbox{0.n}$ $\,X=\bP^2$ and $\sigma$ acts via weights $(1,1,\zeta_n).$ One has $\Br([X/G])=0$ in these cases.

\

\noindent$\bullet$ {\bf Involutions.}
There are three types of involutions (elements of order 2)  in $\Cr_2(\bC)$, up to conjugation. They are:

\

\noindent$\fbox{C.2}$ {\em de Jonqui\`eres involutions}: $X$ is a conic bundle and the fixed locus is a hyperelliptic curve of genus $g>0.$ The model is in standard form. We have
$$
\Br([X/G])=(\bZ/2\bZ)^{2g}
$$
(see \cite{BP}).
\

\noindent$\fbox{2.G}$ {\em Geiser involutions}: 
The model is given by
$$
X=\{w^2=L_4(x,y,z)\}\subset\bP(2,1,1,1)_{w,x,y,z},
$$
$$
  G=\langle\sigma\rangle=C_2,\quad\sigma \colon(w,x,y,z)\mapsto (-w,x,y,z).
$$
The fixed curves stratification is given by
$$
    \centering
    \begin{tabular}{c|c|c|c|c|c|c}

$i$&Curve $\xi_i$& $I_{\xi_i}$&$D_{\xi_i}$&$\rg(\xi_i)$&$\rg(\xi_i/D_{\xi_i})$&Standard form\\\hline
    1&$ \{w=0\}$& $C_2=\langle\sigma\rangle$&$G$&3&3&yes\\\end{tabular}
$$
We have 
$$
\Br([X/G])=(\bZ/2\bZ)^{6}.
$$
   
\

\noindent$\fbox{1.B}$ {\em Bertini involutions}: The model is given by
$$
 X=\{w^2=z^3+L_2(x,y)z^2+L_4(x,y)z+L_6(x,y)\}\subset\bP(3,1,1,2)_{w,x,y,z},
$$
  $$ 
  G=\langle\sigma\rangle=C_2,\quad\sigma \colon(w,x,y,z)\mapsto (-w,x,y,z).
$$
The fixed curves stratification is
$$
    \centering
    \begin{tabular}{c|c|c|c|c|c|c}
$i$&Curve $\xi_i$& $I_{\xi_i}$&$D_{\xi_i}$&$\rg(\xi_i)$&$\rg(\xi_i/D_{\xi_i})$&Standard form\\\hline
    1&$ \{w=0\}$& $C_2=\langle\sigma\rangle$&$G$&4&4&yes\\
    \end{tabular}
$$
We have 
$$
\Br([X/G])=(\bZ/2\bZ)^{8}.
$$

\

\noindent$\bullet$ {\bf Roots of de Jonqui\`eres Involutions.}

\noindent$\fbox{C.ro.m}$ and $\fbox{C.re.m}$ $\,X$ is a conic bundle, $\sigma^m$ is a de Jonqui\`eres involution for some integer $m$ and $n=2m$. The only stratum with nontrivial stabilizer is a hyperelliptic curve $\xi$ of genus $g$ fixed by $\sigma^m$. The model $X\actsfromright G$ is in standard form and $\Br([X/G])$ has been computed in \cite[Section 5]{KT22}. The genus of the quotient curve $\xi/G$ depends on the number $s$ of fixed points of $\sigma$ on $\xi$. Note that $s$ can be 0, 2 or 4. Let $r=\frac{2g+2}{m}$,
we have
$$
\Br([X/G])=\begin{cases}
(\bZ/2\bZ)^{r-2}&\text{if } s=4,\\
(\bZ/2\bZ)^{r-1}&\text{if } s=2,\\
(\bZ/2\bZ)^{r}&\text{if } s=0.\\
\end{cases}
$$

\

\noindent$\bullet$ {\bf n=3.}

\noindent$\fbox{3.3}$ The model is given by
$$
 X=\{w^3=L_3(x,y,z)\}\subset\bP^3_{w,x,y,z},
$$
$$
 G=\langle\sigma\rangle=C_3,\quad \sigma \colon(w,x,y,z)\mapsto (\zeta_3w,x,y,z).
$$
The fixed curves stratification is
$$
    \centering
    \begin{tabular}{c|c|c|c|c|c|c}
$i$&Curve $\xi_i$& $I_{\xi_i}$&$D_{\xi_i}$&$\rg(\xi_i)$&$\rg(\xi_i/D_{\xi_i})$&Standard form\\\hline
    1&$ \{w=0\}$& $C_3=\langle\sigma\rangle$&$G$&1&1&yes\\
    \end{tabular}
$$
We have
$$
\Br([X/G])=(\bZ/3\bZ)^2.
$$

\

\noindent \fbox{1.$\rho$} The model is given by
$$
 X=\{w^2=z^3+L_6(x,y,z)\}\subset\bP(3,1,1,2)_{w,x,y,z},
$$
$$
 G=\langle\sigma\rangle=C_3,\quad \sigma \colon(w,x,y,z)\mapsto (w,x,y,\zeta_3z).
$$
The fixed curves stratification is
$$
    \centering
    \begin{tabular}{c|c|c|c|c|c|c}
$i$&Curve $\xi_i$& $I_{\xi_i}$&$D_{\xi_i}$&$\rg(\xi_i)$&$\rg(\xi_i/D_{\xi_i})$&Standard form\\\hline
    1&$ \{z=0\}$& $C_3=\langle\sigma\rangle$&$G$&2&2&yes\\
    \end{tabular}
$$
We have
$$
\Br([X/G])=(\bZ/3\bZ)^4.
$$

\

\noindent$\bullet$ {\bf n=4.}

\noindent$\fbox{2.4}$ The model is given by
$$
 X=\{w^2=L_4(x,y)+z^4\}\subset\bP(2,1,1,1)_{w,x,y,z},
$$
$$
 G=\langle\sigma\rangle=C_4,\quad\sigma \colon(w,x,y,z)\mapsto (w,x,y,\zeta_4z).
$$
The fixed curves stratification is
$$
    \centering
    \begin{tabular}{c|c|c|c|c|c|c}
$i$&Curve $\xi_i$& $I_{\xi_i}$&$D_{\xi_i}$&$\rg(\xi_i)$&$\rg(\xi_i/D_{\xi_i})$&Standard form\\\hline
    1&$ \{z=0\}$& $C_4=\langle\sigma\rangle$&$G$&1&1&yes\\
    \end{tabular}
$$
We have
$$
\Br([X/G])=(\bZ/4\bZ)^2.
$$

\

\noindent$\fbox{1.B2.2}$ The model is given by
$$
 X=\{w^2=z^3+zL_2(x^2,y^2)+xyL_2'(x^2,y^2)\}\subset\bP(3,1,1,2)_{w,x,y,z},
$$
$$
  G=\langle\sigma\rangle=C_4,\quad\sigma \colon(w,x,y,z)\mapsto (\zeta_4w,x,-y,-z).
$$
The fixed curves stratification is
$$
    \centering
    \begin{tabular}{c|c|c|c|c|c|c}
$i$&Curve $\xi_i$& $I_{\xi_i}$&$D_{\xi_i}$&$\rg(\xi_i)$&$\rg(\xi_i/D_{\xi_i})$&Standard form\\\hline
    1&$ \{w=0\}$& $C_2=\langle\sigma^2\rangle$&$G$&4&2&yes\\
    \end{tabular}
$$
We have
$$
\Br([X/G])=(\bZ/2\bZ)^4.
$$

\

\noindent$\bullet$ {\bf n=5.}

\noindent$\fbox{1.5}$ The model is given by
$$
 X=\{w^2=z^3+\lambda x^4z+x(\mu x^5+y^5)\}\subset\bP(3,1,1,2)_{w,x,y,z},
$$
$$
  G=\langle\sigma\rangle=C_5,\quad\sigma \colon(w,x,y,z)\mapsto (w,x,\zeta_5y,z).
$$
The fixed curves stratification is
$$
    \centering
    \begin{tabular}{c|c|c|c|c|c|c}
$i$&Curve $\xi_i$& $I_{\xi_i}$&$D_{\xi_i}$&$\rg(\xi_i)$&$\rg(\xi_i/D_{\xi_i})$&Standard form\\\hline
    1&$ \{y=0\}$& $C_5=\langle\sigma\rangle$&$G$&1&1&yes\\
    \end{tabular}
$$
We have
$$
\Br([X/G])=(\bZ/5\bZ)^2.
$$

\

\noindent$\bullet$ {\bf n=6.}

\noindent$\fbox{3.6.1}$ The model is given by
$$
 X=\{w^3+x^3+y^3+xz^2+\lambda yz^2=0\}\subset\bP^3_{w,x,y,z},
$$
$$
  G=\langle\sigma\rangle=C_6,\quad\sigma \colon(w,x,y,z)\mapsto (\zeta_3w,x,y,-z).
$$
The fixed curves stratification is
$$
    \centering
    \begin{tabular}{c|c|c|c|c|c|c}
$i$&Curves $\xi_i$& $I_{\xi_i}$&$D_{\xi_i}$&$\rg(\xi_i)$&$\rg(\xi_i/D_{\xi_i})$&Standard form\\\hline
    1&$ \{z=0\}$& $C_2=\langle\sigma^3\rangle$&$G$&1&0&\multirow{2}{*}{yes}\\\cline{1-6}
    2&$ \{w=0\}$& $C_3=\langle\sigma^2\rangle$&$G$&1&0&\\
    \end{tabular}
$$
We have
$$
\Br([X/G])=0.
$$

\

\noindent$\fbox{3.6.2}$ The model is given by
$$
 X=\{wx^2+w^3+y^3+z^3+\lambda wyz=0\}\subset\bP^3_{w,x,y,z},
$$
$$
  G=\langle\sigma\rangle=C_6,\quad\sigma \colon(w,x,y,z)\mapsto (w,-x,\zeta_3y,\zeta_3^2z).
$$
The fixed curves stratification is
$$
    \centering
    \begin{tabular}{c|c|c|c|c|c|c}
$i$&Curve $\xi_i$& $I_{\xi_i}$&$D_{\xi_i}$&$\rg(\xi_i)$&$\rg(\xi_i/D_{\xi_i})$&Standard form\\\hline
    1&$ \{x=0\}$& $C_2=\langle\sigma^3\rangle$&$G$&1&1&yes\\
    \end{tabular}
$$
We have
$$
\Br([X/G])=(\bZ/2\bZ)^2.
$$

\

\noindent$\fbox{2.G3.1}$ The model is given by
$$
 X=\{w^2=L_4(x,y)+z^3L_1(x,y)\}\subset\bP(2,1,1,1)_{w,x,y,z},
$$
$$
  G=\langle\sigma\rangle=C_6,\quad\sigma \colon(w,x,y,z)\mapsto (-w,x,y,\zeta_3z).
$$
The fixed curves stratification is
$$
    \centering
    \begin{tabular}{c|c|c|c|c|c|c}
$i$&Curves $\xi_i$& $I_{\xi_i}$&$D_{\xi_i}$&$\rg(\xi_i)$&$\rg(\xi_i/D_{\xi_i})$&Standard form\\\hline
    1&$ \{w=0\}$& $C_2=\langle\sigma^3\rangle$&$G$&3&0&\multirow{2}{*}{yes}\\\cline{1-6}
    2&$ \{z=0\}$& $C_3=\langle\sigma^2\rangle$&$G$&1&0&\\
    \end{tabular}
$$
We have
$$
\Br([X/G])=0.
$$

\

\noindent$\fbox{2.G3.2}$ The model is given by
$$
 X=\{w^2=x(x^3+y^3+z^3)+yzL_1(x^2,yz)\}\subset\bP(2,1,1,1)_{w,x,y,z},
$$
$$
  G=\langle\sigma\rangle=C_6,\quad\sigma \colon(w,x,y,z)\mapsto (-w,x,\zeta_3y,\zeta_3^2z).
$$
The fixed curves stratification is
$$
    \centering
    \begin{tabular}{c|c|c|c|c|c|c}
$i$&Curve $\xi_i$& $I_{\xi_i}$&$D_{\xi_i}$&$\rg(\xi_i)$&$\rg(\xi_i/D_{\xi_i})$&Standard form\\\hline
    1&$ \{w=0\}$& $C_2=\langle\sigma^3\rangle$&$G$&3&1&yes\\
    \end{tabular}
$$
We have
$$
\Br([X/G])=(\bZ/2\bZ)^2.
$$

\

\noindent$\fbox{2.6}$ The model is given by
$$
 X=\{w^2=x^3y+y^4+z^4+\lambda y^2z^2\}\subset\bP(2,1,1,1)_{w,x,y,z},
$$
$$
  G=\langle\sigma\rangle=C_6,\quad\sigma \colon(w,x,y,z)\mapsto (-w,\zeta_3x,y,-z).
$$
The fixed curves stratification is
$$
    \centering
    \begin{tabular}{c|c|c|c|c|c|c}
$i$&Curve $\xi_i$& $I_{\xi_i}$&$D_{\xi_i}$&$\rg(\xi_i)$&$\rg(\xi_i/D_{\xi_i})$&Standard form\\\hline
    1&$ \{x=0\}$& $C_3=\langle\sigma^2\rangle$&$G$&1&1&yes\\
    \end{tabular}
$$
We have
$$
\Br([X/G])=(\bZ/3\bZ)^2.
$$

\

\noindent \fbox{1.$\sigma\rho$} The model is given by
$$
 X=\{w^2=z^3+L_6(x,y)\}\subset\bP(3,1,1,2)_{w,x,y,z},
$$
$$
  G=\langle\sigma\rangle=C_6,\quad\sigma \colon(w,x,y,z)\mapsto (-w,x,y,\zeta_3z).
$$
The fixed curves stratification is
$$
    \centering
    \begin{tabular}{c|c|c|c|c|c|c}
$i$&Curves $\xi_i$& $I_{\xi_i}$&$D_{\xi_i}$&$\rg(\xi_i)$&$\rg(\xi_i/D_{\xi_i})$&Standard form\\\hline
    1&$ \{w=0\}$& $C_2=\langle\sigma^3\rangle$&$G$&4&0&\multirow{2}{*}{yes}\\\cline{1-6}
    2&$ \{z=0\}$& $C_3=\langle\sigma^2\rangle$&$G$&2&0&
    \end{tabular}
$$
We have
$$
\Br([X/G])=0.
$$

\

\noindent \fbox{1.$\rho2$} The model is given by
$$
 X=\{w^2=z^3+L_3(x^2,y^2)\}\subset\bP(3,1,1,2)_{w,x,y,z},
$$
$$
  G=\langle\sigma\rangle=C_6,\quad\sigma \colon(w,x,y,z)\mapsto (w,x,-y,\zeta_3z).
$$
The fixed curves stratification is
$$
    \centering
    \begin{tabular}{c|c|c|c|c|c|c}
$i$&Curves $\xi_i$& $I_{\xi_i}$&$D_{\xi_i}$&$\rg(\xi_i)$&$\rg(\xi_i/D_{\xi_i})$&Standard form\\\hline
    1&$ \{y=0\}$& $C_2=\langle\sigma^3\rangle$&$G$&1&0&\multirow{2}{*}{yes}\\\cline{1-6}
    2&$ \{z=0\}$& $C_3=\langle\sigma^2\rangle$&$G$&2&0&
    \end{tabular}
$$
We have
$$
\Br([X/G])=0.
$$

\

\noindent \fbox{1.B3.1} The model is given by
$$
 X=\{w^2=z^3+xL_1(x^3,y^3)z+L_2(x^3,y^3)\}\subset\bP(3,1,1,2)_{w,x,y,z},
$$
$$
  G=\langle\sigma\rangle=C_6,\quad\sigma \colon(w,x,y,z)\mapsto (-w,x,\zeta_3y,z).
$$
The fixed curves stratification is
$$
    \centering
    \begin{tabular}{c|c|c|c|c|c|c}
$i$&Curves $\xi_i$& $I_{\xi_i}$&$D_{\xi_i}$&$\rg(\xi_i)$&$\rg(\xi_i/D_{\xi_i})$&Standard form\\\hline
    1&$ \{w=0\}$& $C_2=\langle\sigma^3\rangle$&$G$&4&0&\multirow{2}{*}{yes}\\\cline{1-6}
    2&$ \{y=0\}$& $C_3=\langle\sigma^2\rangle$&$G$&1&0&
    \end{tabular}
$$
We have
$$
\Br([X/G])=0.
$$

\

\noindent \fbox{1.B3.2} The model is given by
$$
 X=\{w^2=z^3+\lambda x^2y^2z+L_2(x^3,y^3)\}\subset\bP(3,1,1,2)_{w,x,y,z},
$$
$$
  G=\langle\sigma\rangle=C_6,\quad\sigma \colon(w,x,y,z)\mapsto (-w,x,\zeta_3y,\zeta_3 z).
$$
The fixed curves stratification is
$$
    \centering
    \begin{tabular}{c|c|c|c|c|c|c}
$i$&Curve $\xi_i$& $I_{\xi_i}$&$D_{\xi_i}$&$\rg(\xi_i)$&$\rg(\xi_i/D_{\xi_i})$&Standard form\\\hline
    1&$ \{w=0\}$& $C_2=\langle\sigma^3\rangle$&$G$&4&2&{yes}\\
    \end{tabular}
$$
We have
$$
\Br([X/G])=(\bZ/2\bZ)^4.
$$

\

\noindent \fbox{1.6} The model is given by
$$
 X=\{w^2=z^3+\lambda x^4z+\mu x^6+y^6\}\subset\bP(3,1,1,2)_{w,x,y,z},
$$
$$
  G=\langle\sigma\rangle=C_6,\quad\sigma \colon(w,x,y,z)\mapsto (w,x,-\zeta_3y,z).
$$
The fixed curves stratification is
$$
    \centering
    \begin{tabular}{c|c|c|c|c|c|c}
$i$&Curve $\xi_i$& $I_{\xi_i}$&$D_{\xi_i}$&$\rg(\xi_i)$&$\rg(\xi_i/D_{\xi_i})$&Standard form\\\hline
    1&$ \{w=0\}$& $C_6=\langle\sigma\rangle$&$G$&1&1&{yes}\\
    \end{tabular}
$$
We have
$$
\Br([X/G])=(\bZ/6\bZ)^2.
$$

\

\noindent$\bullet$ {\bf n=8.}\\

\noindent  \fbox{1.B4.2} The model is given by
$$
 X=\{w^2=\lambda x^2y^2z+xy(x^4+y^4)\}\subset\bP(3,1,1,2)_{w,x,y,z},
$$
$$
  G=\langle\sigma\rangle=C_8,\quad\sigma \colon(w,x,y,z)\mapsto (\zeta_8w,x,\zeta_4y,-\zeta_4z).
$$
The fixed curves stratification is
$$
    \centering
    \begin{tabular}{c|c|c|c|c|c}
$i$&Curves $\xi_i$& $I_{\xi_i}$&$D_{\xi_i}$&$\rg(\xi_i)$&Standard form\\\hline
    1&$ \{x=0\}$& $C_2=\langle\sigma^4\rangle$&$G$&0&\multirow{3}{*}{no}\\\cline{1-5}
    2&$ \{y=0\}$& $C_2=\langle\sigma^4\rangle$&$G$&0&\\\cline{1-5}
   3&$ \{\lambda xyz+x^4+y^4=0\}$& $C_2=\langle\sigma^4\rangle$&$G$&0&\\
    \end{tabular}
$$
The model is not in standard form: the divisor $\{w=0\}\cap X$ fixed by $\sigma^4$ is the union of three rational curves $\xi_1, \xi_2$ and $\xi_3$ meeting at one point $p=[0:0:0:1]$, and thus not normal crossing. Moreover, $p$ is a node of $\xi_3$. To reach a standard form, consider the blowup of $X$ at $p$. Let $E_1$ be the exceptional divisor and $\tilde\xi_i$ be the strict transform of $\xi_i$ for $i=1,2,3.$ We find $\tilde \xi_3$ meets $E_1$ at two points $p_1$ and $p_2$, $\tilde \xi_1\cap \tilde \xi_3\cap E_1=\{p_1\}$, $\tilde \xi_2\cap\tilde \xi_3\cap E_1=\{p_2\}$ and $\tilde \xi_1$ and $\tilde\xi_2$ are disjoint. Then blowing up the points $p_1$ and $p_2$ brings the model into a standard form. One then computes via Proposition~\ref{prop:ratsurf} that
$$
\Br([X/G])=\bZ/2\bZ.
$$


\

\noindent$\bullet$ {\bf n=9.}\\

\noindent  \fbox{3.9} The model is given by
$$
 X=\{w^3+xz^2+x^2y+y^2z=0\}\subset\bP^3_{w,x,y,z},
$$
$$
  G=\langle\sigma\rangle=C_9,\quad\sigma \colon(w,x,y,z)\mapsto (\zeta_9w,x,\zeta_3y,\zeta_3^2z).
$$
The fixed curves stratification is given by 
$$
    \centering
    \begin{tabular}{c|c|c|c|c|c|c}
$i$&Curve $\xi_i$& $I_{\xi_i}$&$D_{\xi_i}$&$\rg(\xi_i)$&$\rg(\xi_i/D_{\xi_i})$&Standard form\\\hline
    1&$ \{w=0\}$& $C_3=\langle\sigma^3\rangle$&$G$&1&0&{yes}\\
    \end{tabular}
$$
We have
$$
\Br([X/G])=0.
$$

\

\noindent$\bullet$ {\bf n=10.}\\

\noindent  \fbox{1.B5} The model is given by
$$
 X=\{w^2=z^3+\lambda x^4z+x(\mu x^5+y^5)\}\subset\bP(3,1,1,2)_{w,x,y,z},
$$
$$
  G=\langle\sigma\rangle=C_{10},\quad\sigma \colon(w,x,y,z)\mapsto (-w,x,\zeta_5y,z).
$$
The fixed curves stratification is given by 
$$
    \centering
    \begin{tabular}{c|c|c|c|c|c|c}
$i$&Curves $\xi_i$& $I_{\xi_i}$&$D_{\xi_i}$&$\rg(\xi_i)$&$\rg(\xi_i/D_{\xi_i})$&Standard form\\\hline
    1&$ \{w=0\}$& $C_2=\langle\sigma^5\rangle$&$G$&4&0&\multirow{2}{*}{yes}\\\cline{1-6}
    2&$ \{y=0\}$& $C_{5}=\langle\sigma^2\rangle$&$G$&1&0&
    \end{tabular}
$$
We have
$$
\Br([X/G])=0.
$$

\noindent$\bullet$ {\bf n=12.}\\

\noindent  \fbox{3.12} The model is given by
$$
 X=\{w^3+x^3+yz^2+y^2x=0\}\subset\bP^3_{w,x,y,z},
$$
$$
  G=\langle\sigma\rangle=C_{12},\quad\sigma \colon(w,x,y,z)\mapsto (\zeta_3w,x,-y,\zeta_4z).
$$
The fixed curves stratification is given by 
$$
    \centering
    \begin{tabular}{c|c|c|c|c|c|c}
$i$&Curves $\xi_i$& $I_{\xi_i}$&$D_{\xi_i}$&$\rg(\xi_i)$&$\rg(\xi_i/D_{\xi_i})$&Standard form\\\hline
    1&$ \{z=0\}$& $C_2=\langle\sigma^6\rangle$&$G$&1&0&\multirow{2}{*}{yes}\\\cline{1-6}
    2&$ \{w=0\}$& $C_{3}=\langle\sigma^4\rangle$&$G$&1&0&
    \end{tabular}
$$
We have
$$
\Br([X/G])=0.
$$

\

\noindent  \fbox{2.12} The model is given by
$$
 X=\{w^2=x^3y+y^4+z^4\}\subset\bP(2,1,1,1)_{w,x,y,z},
$$
$$
  G=\langle\sigma\rangle=C_{12},\quad\sigma \colon(w,x,y,z)\mapsto (w,\zeta_3x,y,\zeta_4z).
$$
The fixed curves stratification is given by 
$$
    \centering
    \begin{tabular}{c|c|c|c|c|c|c}
$i$&Curves $\xi_i$& $I_{\xi_i}$&$D_{\xi_i}$&$\rg(\xi_i)$&$\rg(\xi_i/D_{\xi_i})$&Standard form\\\hline
    1&$ \{z=0\}$& $C_2=\langle\sigma^6\rangle$&$G$&1&0&\multirow{2}{*}{yes}\\\cline{1-6}
    2&$ \{x=0\}$& $C_{3}=\langle\sigma^4\rangle$&$G$&1&0&
    \end{tabular}
$$
We have
$$
\Br([X/G])=0.
$$

\

\noindent  \fbox{1.$\sigma\rho$2.2} The model is given by
$$
 X=\{w^2=z^3+xyL_2(x^2,y^2)\}\subset\bP(3,1,1,2)_{w,x,y,z},
$$
$$
  G=\langle\sigma\rangle=C_{12},\quad\sigma \colon(w,x,y,z)\mapsto (\zeta_4w,x,-y,-\zeta_3z).
$$
The fixed curves stratification is given by 
$$
    \centering
    \begin{tabular}{c|c|c|c|c|c|c}
$i$&Curves $\xi_i$& $I_{\xi_i}$&$D_{\xi_i}$&$\rg(\xi_i)$&$\rg(\xi_i/D_{\xi_i})$&Standard form\\\hline
    1&$ \{w=0\}$& $C_2=\langle\sigma^6\rangle$&$G$&4&0&\multirow{2}{*}{yes}\\\cline{1-6}
    2&$ \{z=0\}$& $C_{3}=\langle\sigma^4\rangle$&$G$&2&0&
    \end{tabular}
$$

We have
$$
\Br([X/G])=0.
$$

\

\noindent$\bullet$ {\bf n=14.}\\

\noindent  \fbox{2.G7} The model is given by
$$
 X=\{w^2=x^3y+y^3z+xz^3\}\subset\bP(2,1,1,1)_{w,x,y,z},
$$
$$
  G=\langle\sigma\rangle=C_{14},\quad\sigma \colon(w,x,y,z)\mapsto (-w,\zeta_7x,\zeta_7^4y,\zeta_7^2z).
$$
The fixed curves stratification is given by 
$$
    \centering
    \begin{tabular}{c|c|c|c|c|c|c}
$i$&Curves $\xi_i$& $I_{\xi_i}$&$D_{\xi_i}$&$\rg(\xi_i)$&$\rg(\xi_i/D_{\xi_i})$&Standard form\\\hline
    1&$ \{w=0\}$& $C_2=\langle\sigma^7\rangle$&$G$&3&0&{yes}\\
    \end{tabular}
$$
We have
$$
\Br([X/G])=0.
$$

\

\noindent$\bullet$ {\bf n=15.}\\

\noindent  \fbox{1.$\rho5$} The model is given by
$$
 X=\{w^2=z^3+x(x^5+y^5)\}\subset\bP(3,1,1,2)_{w,x,y,z},
$$
$$
  G=\langle\sigma\rangle=C_{15},\quad\sigma \colon(w,x,y,z)\mapsto (w,x,\zeta_5y,\zeta_3z).
$$
The fixed curves stratification is given by 
$$
    \centering
    \begin{tabular}{c|c|c|c|c|c|c}
$i$&Curves $\xi_i$& $I_{\xi_i}$&$D_{\xi_i}$&$\rg(\xi_i)$&$\rg(\xi_i/D_{\xi_i})$&Standard form\\\hline
    1&$ \{z=0\}$& $C_3=\langle\sigma^5\rangle$&$G$&2&0&\multirow{2}{*}{yes}\\\cline{1-6}
    2&$ \{y=0\}$& $C_{5}=\langle\sigma^3\rangle$&$G$&1&0&
    \end{tabular}
$$

We have
$$
\Br([X/G])=0.
$$

\noindent$\bullet$ {\bf n=18.}\\

\noindent  \fbox{2.G9} The model is given by
$$
 X=\{w^2=x^3y+y^4+xz^3\}\subset\bP(2,1,1,1)_{w,x,y,z},
$$
$$
  G=\langle\sigma\rangle=C_{18},\quad\sigma \colon(w,x,y,z)\mapsto (-w,\zeta_9^6x,y,\zeta_9z).
$$
The fixed curves stratification is given by 
$$
    \centering
    \begin{tabular}{c|c|c|c|c|c|c}
$i$&Curves $\xi_i$& $I_{\xi_i}$&$D_{\xi_i}$&$\rg(\xi_i)$&$\rg(\xi_i/D_{\xi_i})$&Standard form\\\hline
    1&$ \{w=0\}$& $C_2=\langle\sigma^9\rangle$&$G$&3&0&\multirow{2}{*}{yes}\\\cline{1-6}
    2&$ \{z=0\}$& $C_{3}=\langle\sigma^6\rangle$&$G$&1&0&
    \end{tabular}
$$
We have
$$
\Br([X/G])=0.
$$

\

\noindent$\bullet$ {\bf n=20.}\\

\noindent  \fbox{1.B10} The model is given by
$$
 X=\{w^2=z^3+x^4z+xy^5\}\subset\bP(3,1,1,2)_{w,x,y,z},
$$
$$
  G=\langle\sigma\rangle=C_{20},\quad\sigma \colon(w,x,y,z)\mapsto (\zeta_4w,x,\zeta_{10}y,-z).
$$
The fixed curves stratification is given by 
$$
    \centering
    \begin{tabular}{c|c|c|c|c|c|c}
$i$&Curves $\xi_i$& $I_{\xi_i}$&$D_{\xi_i}$&$\rg(\xi_i)$&$\rg(\xi_i/D_{\xi_i})$&Standard form\\\hline
    1&$ \{z=0\}$& $C_2=\langle\sigma^{10}\rangle$&$G$&0&0&\multirow{2}{*}{no}\\\cline{1-6}
    2&$ \{y=0\}$& $C_{5}=\langle\sigma^4\rangle$&$G$&1&0&
    \end{tabular}
$$
The model is not in standard form. The curve $\xi_1$ has an $\sA_4$-singularity at $p=[0:1:0:0]$, and $\xi_1$ intersects $\xi_2$ at $p$ non-transversally. One can obtain a standard form via successive blowups such that the strict transforms of $\xi_1$ and $\xi_2$ and the exceptional divisors form a tree of rational curves. 
We have 
$$
\Br([X/G])=0.
$$

\

\noindent$\bullet$ {\bf n=24.}\\

\noindent  \fbox{1.$\sigma\rho4$} The model is given by
$$
 X=\{w^2=z^3+xy(x^4+y^4)\}\subset\bP(3,1,1,2)_{w,x,y,z},
$$
$$
  G=\langle\sigma\rangle=C_{24},\quad\sigma \colon(w,x,y,z)\mapsto (\zeta_8w,x,\zeta_{4}y,-\zeta_{12}^7z).
$$
The fixed curves stratification is given by 
$$
    \centering
    \begin{tabular}{c|c|c|c|c|c|c}
$i$&Curves $\xi_i$& $I_{\xi_i}$&$D_{\xi_i}$&$\rg(\xi_i)$&$\rg(\xi_i/D_{\xi_i})$&Standard form\\\hline
    1&$ \{w=0\}$& $C_2=\langle\sigma^{12}\rangle$&$G$&4&0&\multirow{2}{*}{yes}\\\cline{1-6}
    2&$ \{z=0\}$& $C_{3}=\langle\sigma^8\rangle$&$G$&2&0&
    \end{tabular}
$$
We have 
$$
\Br([X/G])=0.
$$

\

\noindent$\bullet$ {\bf n=30.}\\

\noindent  \fbox{1.$\sigma\rho5$} The model is given by
$$
 X=\{w^2=z^3+x(x^5+y^5)\}\subset\bP(3,1,1,2)_{w,x,y,z},
$$
$$
  G=\langle\sigma\rangle=C_{30},\quad\sigma \colon(w,x,y,z)\mapsto (-w,x,\zeta_{5}y,\zeta_{3}z).
$$
The fixed curves stratification is given by 
$$
    \centering
    \begin{tabular}{c|c|c|c|c|c|c}
$i$&Curves $\xi_i$& $I_{\xi_i}$&$D_{\xi_i}$&$\rg(\xi_i)$&$\rg(\xi_i/D_{\xi_i})$&Standard form\\\hline
    1&$ \{w=0\}$& $C_2=\langle\sigma^{15}\rangle$&$G$&4&0&\multirow{3}{*}{yes}\\\cline{1-6}
    2&$ \{z=0\}$& $C_{3}=\langle\sigma^{10}\rangle$&$G$&2&0&\\\cline{1-6}
    3&$ \{y=0\}$& $C_{5}=\langle\sigma^6\rangle$&$G$&1&0&
    \end{tabular}
$$
We have 
$$
\Br([X/G])=0.
$$

\

\subsection{Noncyclic groups}

We continue with actions of noncyclic groups.

\

\noindent$\bullet$ {\bf Automorphisms of $\bP^1\times\bP^1$}\\
Let $X=\bP^1\times\bP^1$. Note that
$\rH^1(G,\Pic(X))=0$ for any $G\subset\Aut(X)$ and thus
$$
\Br([X/G])=\rH^2(G,\bC^\times)/\mathrm{Am}(X,G),
$$
where $\mathrm{Am}(X,G)$ is the Amitsur group of the $G$-action on $X$. The computation of $\mathrm{Am}(X,G)$ in this case is straightforward, see e.g., \cite[Proposition 6.7]{amitsur}. So for actions on quadric surfaces, we compute $\Br([X/G])$ from the Amitsur groups. 


\

\noindent  \fbox{0.mn} The action on $X=\bP^1\times\bP^1$ is given by 
$$
G=C_n\times C_m,\qquad (x,y)\stackrel{\sigma_1}{\mapsto}(\zeta_nx,y), \quad(x,y)\stackrel{\sigma_2}{\mapsto}(x,\zeta_my).
$$
One has $\Pic(X)^G=\bZ^2$, generated by the $G$-invariant line bundles $\cO(1,0)$ and $\cO(0,1)$. Both $\cO(1,0)$ and $\cO(0,1)$ are $G$-linearizable. It follows that $\Am(X,G)=0$ and 
$$
\Br([X/G])=\rH^2(G,\bC^\times)=\bZ/\mathrm{gcd}(n,m)\bZ.
$$

\noindent We also compute the fixed curves stratification in this case
$$
    \centering
    \begin{tabular}{c|c|c|c|c|c|c}
$i$&Curves $\xi_i$& $I_{\xi_i}$&$D_{\xi_i}$&$\rg(\xi_i)$&$\rg(\xi_i/D_{\xi_i})$&Standard form\\\hline
    1&$ [1:0]\times \bP^1$& $C_n=\langle\sigma_1\rangle$&$G$&0&0&\multirow{4}{*}{yes}\\\cline{1-6}
     2&$ [0:1]\times \bP^1$& $C_n=\langle\sigma_1\rangle$&$G$&0&0&\\\cline{1-6}
    3&$ \bP^1\times[1:0]$& $C_m=\langle\sigma_2\rangle$&$G$&0&0&\\\cline{1-6}
    4&$ \bP^1\times[0:1]$& $C_m=\langle\sigma_2\rangle$&$G$&0&0&
    \end{tabular}
$$
Using Proposition~\ref{prop:ratsurf}, one can also deduce $\Br([X/G])=\bZ/\mathrm{gcd}(n,m)\bZ$.

\

\noindent  \fbox{P1.22n} The action on $X=\bP^1\times\bP^1$ is given by 
$$
G=C_2\times C_{2n},\qquad (x,y)\mapsto(x^{-1},y), \quad(x,y)\mapsto(-x,\zeta_{2n}y).
$$
One has $\Pic(X)^G=\bZ^2$, generated by $\cO(1,0)$ and $\cO(0,1)$. The line bundle $\cO(1,0)$ is not $G$-linearizable while $\cO(0,1)$ is. It follows that $\Am(X,G)=\bZ/2\bZ$ and
$$
\Br([X/G])=0.
$$

\

\noindent  \fbox{P1.222n} The action on $X=\bP^1\times\bP^1$ is given by 
$$
G=C_2^2\times C_{2n},\qquad (x,y)\mapsto(\pm x^{\pm1},y), \quad(x,y)\mapsto(x,\zeta_{2n}y).
$$
One has $\Pic(X)^G=\bZ^2$, generated by $\cO(1,0)$ and $\cO(0,1)$. The line bundle $\cO(1,0)$ is not $G$-linearizable while $\cO(0,1)$ is. It follows that $\Am(X,G)=\bZ/2\bZ$ and
$$
\Br([X/G])=(\bZ/2\bZ)^2.
$$

\

\noindent  \fbox{P1.22.1} The action on $X=\bP^1\times\bP^1$ is given by 
$$
G=C_2^2,\qquad (x,y)\mapsto(\pm x^{\pm1},y).
$$
One has $\Pic(X)^G=\bZ^2$, generated by $\cO(1,0)$ and $\cO(0,1)$. The line bundle $\cO(1,0)$ is not $G$-linearizable while $\cO(0,1)$ is. It follows that $\Am(X,G)=\bZ/2\bZ$ and
$$
\Br([X/G])=0.
$$

\

\noindent  \fbox{P1.222} The action on $X=\bP^1\times\bP^1$ is given by 
$$
G=C_2^3,\qquad (x,y)\mapsto(\pm x,\pm y),\qquad (x,y)\mapsto( x^{-1},y).
$$
One has $\Pic(X)^G=\bZ^2$, generated by $\cO(1,0)$ and $\cO(0,1)$. The line bundle $\cO(1,0)$ is not $G$-linearizable while $\cO(0,1)$ is. It follows that $\Am(X,G)=\bZ/2\bZ$ and
$$
\Br([X/G])=(\bZ/2\bZ)^2.
$$

\

\noindent  \fbox{P1.2222} The action on $X=\bP^1\times\bP^1$ is given by 
$$
G=C_2^4,\qquad (x,y)\mapsto(\pm x^{\pm1},\pm y^{\pm1}).
$$
One has $\Pic(X)^G=\bZ^2$, generated by $\cO(1,0)$ and $\cO(0,1)$. Both $\cO(1,0)$ and $\cO(0,1)$ are not $G$-linearizable. It follows that $\Am(X,G)=(\bZ/2\bZ)^2$ and
$$
\Br([X/G])=(\bZ/2\bZ)^4.
$$

\

\noindent  \fbox{P1s.24} The action on $X=\bP^1\times\bP^1$ is given by 
$$
G=C_2\times C_4,\qquad (x,y)\mapsto(x^{-1}, y^{-1}),\qquad (x,y)\mapsto(-y,x). 
$$
One has $\Pic(X)^G=\bZ$, generated by $\cO(1,1)$, which is not $G$-linearizable. It follows that $\Am(X,G)=\bZ/2\bZ$ and
$$
\Br([X/G])=0.
$$

\

\noindent  \fbox{P1s.222} The action on $X=\bP^1\times\bP^1$ is given by 
$$
G=C_2^3,\qquad (x,y)\mapsto(-x, -y),\quad (x,y)\mapsto(x^{-1}, y^{-1}),\quad (x,y)\mapsto(y,x). 
$$
One has $\Pic(X)^G=\bZ$, generated by $\cO(1,1)$, which is $G$-linearizable. It follows that $\Am(X,G)=0$ and
$$
\Br([X/G])=(\bZ/2\bZ)^3.
$$


\

\noindent$\bullet$ {\bf Automorphisms of $\bP^2$}

\

\noindent  \fbox{0.V9} The action on $X=\bP^2$ is given by 
$$
G=C_3^2,\qquad (x:y:z)\mapsto(x:\zeta_3y:\zeta_3^2z),\quad (x:y:z)\mapsto(y:z:x).
$$
One has $\Pic(X)^G=\bZ$, generated by $\cO(1)$, which is not $G$-linearizable. It follows that $\Am(X,G)=\bZ/3\bZ$ and 
$$
\Br([X/G])=0.
$$

\noindent$\bullet$ {\bf Automorphisms of del Pezzo surfaces of degree $4$}\\

The surface $X\subset\bP^4_{x_1,\ldots,x_5}$ is given by the following equations with general $a,b,c\in\bC$
\begin{align}\label{eq:Xdp4}
    cx_1^2-ax_3^2-(a-c)x_4^2-ac(a-c)x_5^2=0\\
    cx_2^2-bx_3^2+(c-b)x_4^2-bc(c-b)x_5^2=0\notag.
\end{align}

\noindent$\fbox{4.222}$ The action on $X$ is given by 
$$
G=C_2^3,\qquad \sigma_1: (\mathbf x)\mapsto(-x_1,x_2,x_3,x_4,x_5),
$$
$$
\sigma_2: (\mathbf x)\mapsto(x_1,-x_2,x_3,x_4,x_5),\quad \sigma_3: (\mathbf x)\mapsto(x_1,x_2,-x_3,x_4,x_5).
$$
The fixed curves stratification is
$$
    \centering
    \begin{tabular}{c|c|c|c|c|c|c}
$i$&Curves $\xi_i$& $I_{\xi_i}$&$D_{\xi_i}$&$\rg(\xi_i)$&$\rg(\xi_i/D_{\xi_i})$&Standard form\\\hline
    1&$\{x_1=0\}$& $C_2=\langle\sigma_1\rangle$&$G$&1&0&\multirow{3}{*}{yes}\\\cline{1-6}
      2&$\{x_2=0\}$& $C_2=\langle\sigma_2\rangle$&$G$&1&0&\\\cline{1-6}
   3&$\{x_3=0\}$& $C_2=\langle\sigma_3\rangle$&$G$&1&0&\\
    \end{tabular}
$$
The images of $\xi_i$ and $\xi_j$  in $X^{(1)}/G$ intersect in 2 points for $i\ne j\in\{1,2,3\}.$ We find 
$$
\Br([X/G])=(\bZ/2\bZ)^4.
$$

\

\noindent$\fbox{4.2222}$ The action on $X$ is given by 
$$
G=C_2^4,\quad \sigma_1: (\mathbf x)\mapsto(-x_1,x_2,x_3,x_4,x_5),\quad \sigma_2: (\mathbf x)\mapsto(x_1,-x_2,x_3,x_4,x_5),
$$
$$
\sigma_3: (\mathbf x)\mapsto(x_1,x_2,-x_3,x_4,x_5),\quad \sigma_4: (\mathbf x)\mapsto(x_1,x_2,x_3,-x_4,x_5).
$$
The fixed curves stratification is
$$
    \centering
    \begin{tabular}{c|c|c|c|c|c|c}
$i$&Curves $\xi_i$& $I_{\xi_i}$&$D_{\xi_i}$&$\rg(\xi_i)$&$\rg(\xi_i/D_{\xi_i})$&Standard form\\\hline
    1&$\{x_1=0\}$& $C_2=\langle\sigma_1\rangle$&$G$&1&0&\multirow{5}{*}{yes}\\\cline{1-6}
      2&$\{x_2=0\}$& $C_2=\langle\sigma_2\rangle$&$G$&1&0&\\\cline{1-6}
   3&$\{x_3=0\}$& $C_2=\langle\sigma_3\rangle$&$G$&1&0&\\\cline{1-6}
   4&$\{x_4=0\}$& $C_2=\langle\sigma_4\rangle$&$G$&1&0&\\\cline{1-6}
   5&$\{x_5=0\}$& $C_2=\langle\sigma_1\sigma_2\sigma_3\sigma_4\rangle$&$G$&1&0&\\
    \end{tabular}
$$

The images of $\xi_i$ and $\xi_j$  in $X^{(1)}/G$ intersect in 1 point for $i\ne j\in\{1,\ldots,5\}.$ We find 
$$
\Br([X/G])=(\bZ/2\bZ)^6.
$$

\

\noindent$\fbox{4.42}$ The surface $X$ is given  by \eqref{eq:Xdp4} with $(a:b:c)=(1:\xi:1+\xi)$ for any $\xi\in \bC\setminus\{0,\pm1\}$. The $G=C_4\times C_2$-action on $X$ is generated by
$$
\sigma_1: (\mathbf x)\mapsto(-x_2,x_1,x_4,x_3,-x_5),\qquad \sigma_2: (\mathbf x)\mapsto(x_1,x_2,x_3,x_4,-x_5).
$$
The fixed curves stratification is
$$
    \centering
    \begin{tabular}{c|c|c|c|c|c|c}
$i$&Curves $\xi_i$& $I_{\xi_i}$&$D_{\xi_i}$&$\rg(\xi_i)$&$\rg(\xi_i/D_{\xi_i})$&Standard form\\\hline
    1&$\{x_5=0\}$& $C_2=\langle\sigma_2\rangle$&$G$&1&1&{yes}\\
    \end{tabular}
$$
 We have
$$
\Br([X/G])=(\bZ/2)^2.
$$

\

\noindent$\bullet$ {\bf Automorphisms of cubic surfaces}\\

\noindent$\fbox{3.33.1}$ 
The model is given by
$$
 X=\{w^3+x^3+y^3+z^3=0\}\subset\bP^3_{w,x,y,z},\quad
  G=\langle\sigma_1, \sigma_2\rangle=C_{3}^2,
  $$
  $$
  \sigma_1: (w,x,y,z)\mapsto (\zeta_3 w,x,y,z), \quad\sigma_2: (w,x,y,z)\mapsto (w,x,y,\zeta_3 z).
$$
The group $$\Br([X/G])=(\bZ/3\bZ)^2$$ has been computed in \cite[Section 5]{KT22}.

\

\noindent$\fbox{3.33.2}$ The model is given by
$$
 X=\{w^3+x^3+y^3+z^3+\lambda xyz=0\}\subset\bP^3_{w,x,y,z},
\quad
  G=\langle\sigma_1, \sigma_2\rangle=C_{3}^2,
  $$
  $$
  \sigma_1: (w,x,y,z)\mapsto (\zeta_3 w,x,y,z), \quad \sigma_2: (w,x,y,z)\mapsto (w,x,\zeta_3 y,\zeta_3^2 z).
$$
The fixed curves stratification is given by 
$$
    \centering
    \begin{tabular}{c|c|c|c|c|c|c}
$i$&Curves $\xi_i$& $I_{\xi_i}$&$D_{\xi_i}$&$\rg(\xi_i)$&$\rg(\xi_i/D_{\xi_i})$&Standard form\\\hline
    1&$ \{w=0\}$& $C_{3}=\langle\sigma_1\rangle$&$G$&1&1&yes
    \end{tabular}
$$
and
$$
\Br([X/G])=(\bZ/3\bZ)^2.
$$

\

\noindent$\fbox{3.36}$
The model is given by
$$
 X=\{w^3+x^3+xy^2+z^3=0\}\subset\bP^3_{w,x,y,z},\quad
  G=\langle\sigma_1, \sigma_2\rangle=C_{3}\times C_6,
  $$
  $$
  \sigma_{1}:(w,x,y,z)\mapsto (\zeta_3 w,x,y,z),\quad \sigma_2: (w,x,y,z)\mapsto (w,x,-y,\zeta_3 z).
$$
The fixed curves stratification is given by 
$$
    \centering
    \begin{tabular}{c|c|c|c|c|c|c}
$i$&Curves $\xi_i$& $I_{\xi_i}$&$D_{\xi_i}$&$\rg(\xi_i)$&$\rg(\xi_i/D_{\xi_i})$&Standard form\\\hline
  1&$ \{y=0\}$& $C_{2}=\langle\sigma_2^3\rangle$&$G$&1&0&\multirow{3}{*}{yes}\\\cline{1-6}
    2&$ \{w=0\}$& $C_3=\langle\sigma_1\rangle$&$G$&1&0&\\\cline{1-6}
    3&$ \{z=0\}$& $C_3=\langle\sigma_2^2\rangle$&$G$&1&0&
    \end{tabular}
$$
The images of $\xi_2$ and $\xi_3$ in $X^{(1)}/G$ intersect in two points, so that 
$$
\Br([X/G])=\bZ/3\bZ.
$$
\

\noindent$\fbox{3.333}$
The model is given by
$$
 X=\{w^3+x^3+y^3+z^3=0\}\subset\bP^3_{w,x,y,z},
$$  
$$
G=\langle\sigma_1, \sigma_2, \sigma_3\rangle=C_{3}^3,\quad\sigma_1:(w,x,y,z)\mapsto (\zeta_3 w,x,y,z),
$$
$$
\sigma_2:(w,x,y,z)\mapsto (w,x,\zeta_3 y,z),\quad\sigma_3:(w,x,y,z)\mapsto (w,x,y,\zeta_3 z).
$$
The group 
$$
\Br([X/G])=(\bZ/3\bZ)^3
$$
has been computed in \cite[Section 5]{KT22}. 

\

\noindent$\bullet$ {\bf Automorphisms of Del Pezzo surfaces of degree $2$}
\\

\noindent  \fbox{2.G2} The model is given by
$$
 X=\{w^2=L_4(x,y)+L_2(x,y)z^2+z^4\}\subset\bP(2,1,1,1)_{w,x,y,z}, 
$$
$$
G=\langle\sigma_{1},\sigma_2\rangle=C_2^2,
$$
$$
\sigma_1 \colon(w,x,y,z)\mapsto (-w,x,y,z),\quad\sigma_2 \colon(w,x,y,z)\mapsto (w,x,y,-z).
$$
The fixed curves stratification is 
$$
    \centering
    \begin{tabular}{c|c|c|c|c|c|c}
$i$&Curves $\xi_i$& $I_{\xi_i}$&$D_{\xi_i}$&$\rg(\xi_i)$&$\rg(\xi_i/D_{\xi_i})$&Standard form\\\hline
    1&$ \{w=0\}$& $C_2=\langle\sigma_1\rangle$&$G$&3&1&\multirow{2}{*}{yes}\\\cline{1-6}
    2&$ \{z=0\}$& $C_2=\langle\sigma_2\rangle$&$G$&1&0&\\    \end{tabular}
$$
The images of $\xi_1$ and $\xi_2$ in $X^{(1)}/G$ meet at four points. Recall  that a zero-cycle  $\sum_in_iP_i$ of degree $0$ on an elliptic curve is a divisor of a function on the curve if and only if  $\sum n_i[P_i]=0$, where the latter sum is for the group law of the elliptic curve. It follows that we have 
$$
\Br([X/G])=(\bZ/2\bZ)^5.
$$



\

\noindent  \fbox{2.G4.1} The model is given by
$$
 X=\{w^2=L_4(x,y)+z^4\}\subset\bP(2,1,1,1)_{w,x,y,z},  \quad
G=\langle\sigma_{1},\sigma_2\rangle=C_2\times C_4,
$$
$$
\sigma_1 \colon(w,x,y,z)\mapsto (-w,x,y,z),\quad\sigma_2 \colon(w,x,y,z)\mapsto (w,x,y,\zeta_4z).
$$
The fixed curves stratification is 
$$
    \centering
    \begin{tabular}{c|c|c|c|c|c|c}
$i$&Curves $\xi_i$& $I_{\xi_i}$&$D_{\xi_i}$&$\rg(\xi_i)$&$\rg(\xi_i/D_{\xi_i})$&Standard form\\\hline
    1&$ \{w=0\}$& $C_2=\langle\sigma_1\rangle$&$G$&3&0&\multirow{2}{*}{yes}\\\cline{1-6}
    2&$ \{z=0\}$& $C_4=\langle\sigma_2\rangle$&$G$&1&0&\\    \end{tabular}
$$
The images of $\xi_1$ and $\xi_2$ in $X^{(1)}/G$ meet at four points. We have 
$$
\Br([X/G])=(\bZ/2\bZ)^3.
$$

\

\noindent  \fbox{2.G4.2} The model is given by
$$
 X=\{w^2=x^4+y^4+z^4+xyL_1(xy,z^2)\}\subset\bP(2,1,1,1)_{w,x,y,z}, 
 $$
 $$
G=\langle\sigma_{1},\sigma_2\rangle=C_2\times C_4,
$$
$$
\sigma_1 \colon(w,x,y,z)\mapsto (-w,x,y,z),\quad\sigma_2 \colon(w,x,y,z)\mapsto (w,x,-y,\zeta_4z).
$$
The fixed curves stratification is 
$$
    \centering
    \begin{tabular}{c|c|c|c|c|c|c}
$i$&Curves $\xi_i$& $I_{\xi_i}$&$D_{\xi_i}$&$\rg(\xi_i)$&$\rg(\xi_i/D_{\xi_i})$&Standard form\\\hline
    1&$ \{w=0\}$& $C_2=\langle\sigma_1\rangle$&$G$&3&1&\multirow{2}{*}{yes}\\\cline{1-6}
    2&$ \{z=0\}$& $C_2=\langle\sigma_2^2\rangle$&$G$&1&0&\\    \end{tabular}
$$
The images of $\xi_1$ and $\xi_2$ in $X^{(1)}/G$ meet at two points. We have 
$$
\Br([X/G])=(\Z/2\Z)^3.
$$

\

\noindent  \fbox{2.G6} The model is given by
$$
 X=\{w^2=x^3y+y^4+z^4+\lambda y^2z^2\}\subset\bP(2,1,1,1)_{w,x,y,z}, 
 $$
 $$
G=\langle\sigma_{1},\sigma_2\rangle=C_2\times C_6,
$$
$$
\sigma_1 \colon(w,x,y,z)\mapsto (-w,x,y,z),\quad\sigma_2 \colon(w,x,y,z)\mapsto (w,\zeta_3x,y,-z).
$$
The fixed curves stratification is 
$$
    \centering
    \begin{tabular}{c|c|c|c|c|c|c}
$i$&Curves $\xi_i$& $I_{\xi_i}$&$D_{\xi_i}$&$\rg(\xi_i)$&$\rg(\xi_i/D_{\xi_i})$&Standard form\\\hline
    1&$ \{w=0\}$& $C_2=\langle\sigma_1\rangle$&$G$&3&0&\multirow{3}{*}{yes}\\\cline{1-6}
    2&$ \{z=0\}$& $C_2=\langle\sigma_2^3\rangle$&$G$&1&0&\\ \cline{1-6}
     3&$ \{x=0\}$& $C_3=\langle\sigma_2^2\rangle$&$G$&1&0&\\ 
    \end{tabular}
$$
 The images of $\xi_1$ and $\xi_2$ in $X^{(1)}/G$ meet at two points. We have 
$$
\Br([X/G])=\bZ/2\bZ.
$$

\

\noindent  \fbox{2.G8} The model is given by
$$
 X=\{w^2=x^3y+xy^3+z^4\}\subset\bP(2,1,1,1)_{w,x,y,z}, 
\quad
G=\langle\sigma_{1},\sigma_2\rangle=C_2\times C_8,
$$
$$
\sigma_1 \colon(w,x,y,z)\mapsto (-w,x,y,z),\quad\sigma_2 \colon(w,x,y,z)\mapsto (w,\zeta_8x,-\zeta_8y,z).
$$
The fixed curves stratification is 
$$
    \centering
    \begin{tabular}{c|c|c|c|c|c|c}
$i$&Curves $\xi_i$& $I_{\xi_i}$&$D_{\xi_i}$&$\rg(\xi_i)$&$\rg(\xi_i/D_{\xi_i})$&Standard form\\\hline
    1&$ \{w=0\}$& $C_2=\langle\sigma_1\rangle$&$G$&3&0&\multirow{2}{*}{yes}\\\cline{1-6}
    2&$ \{z=0\}$& $C_4=\langle\sigma_1\sigma_2^2\rangle$&$G$&1&0&\\ 
    \end{tabular}
$$
The images of $\xi_1$ and $\xi_2$ in $X^{(1)}/G$ meet at three points. We have
$$
\Br([X/G])=(\bZ/2\bZ)^2.
$$

\

\noindent  \fbox{2.G12} The model is given by
$$
 X=\{w^2=x^3y+y^4+z^4\}\subset\bP(2,1,1,1)_{w,x,y,z}, 
\quad
G=\langle\sigma_{1},\sigma_2\rangle=C_2\times C_{12},
$$
$$
\sigma_1 \colon(w,x,y,z)\mapsto (-w,x,y,z),\quad\sigma_2 \colon(w,x,y,z)\mapsto (w,\zeta_3x,y,\zeta_4z).
$$
The fixed curves stratification is 
$$
    \centering
    \begin{tabular}{c|c|c|c|c|c|c}
$i$&Curves $\xi_i$& $I_{\xi_i}$&$D_{\xi_i}$&$\rg(\xi_i)$&$\rg(\xi_i/D_{\xi_i})$&Standard form\\\hline
    1&$ \{w=0\}$& $C_2=\langle\sigma_1\rangle$&$G$&3&0&\multirow{3}{*}{yes}\\\cline{1-6}
    2&$ \{x=0\}$& $C_3=\langle\sigma_2^4\rangle$&$G$&1&0&\\\cline{1-6}
        3&$ \{z=0\}$& $C_4=\langle\sigma_2^3\rangle$&$G$&1&0&\\ 
    \end{tabular}
$$
The images of $\xi_1$ and $\xi_3$ in $X^{(1)}/G$ meet at two points. We have
$$
\Br([X/G])=\bZ/2\bZ.
$$

\

\noindent  \fbox{2.G22} The model is given by
$$
 X=\{w^2=L_2(x^2,y^2,z^2)\}\subset\bP(2,1,1,1)_{w,x,y,z}, 
\quad
G=\langle\sigma_{1},\sigma_2,\sigma_3\rangle=C_2^3,
$$
$$
\sigma_1 \colon(w,x,y,z)\mapsto (-w,x,y,z),\quad\sigma_2 \colon(w,x,y,z)\mapsto (w,x,-y,z),
$$
$$
\sigma_3 \colon(w,x,y,z)\mapsto (w,x,y,-z).
$$
The fixed curves stratification is 
$$
    \centering
    \begin{tabular}{c|c|c|c|c|c|c}
$i$&Curves $\xi_i$& $I_{\xi_i}$&$D_{\xi_i}$&$\rg(\xi_i)$&$\rg(\xi_i/D_{\xi_i})$&Standard form\\\hline
    1&$ \{w=0\}$& $C_2=\langle\sigma_1\rangle$&$G$&3&0&\multirow{4}{*}{yes}\\\cline{1-6}
    2&$ \{x=0\}$& $C_2=\langle\sigma_2\sigma_3\rangle$&$G$&1&0&\\\cline{1-6}
        3&$ \{y=0\}$& $C_2=\langle\sigma_2\rangle$&$G$&1&0&\\ \cline{1-6}
        4&$ \{z=0\}$& $C_2=\langle\sigma_3\rangle$&$G$&1&0&\\ 
    \end{tabular}
$$
Let $p_{ij}$ be the number of intersection points of the images of $\xi_i$ and $\xi_j$ in $X^{(1)}/G$. We record 
$$
p_{12}=p_{13}=p_{14}=2,\quad p_{23}=p_{24}=p_{34}=1.
$$
Since the intersection of any three of the four curves $\xi_i$, $i=1,2,3,4$ is empty, we have
$$
\Br([X/G])=(\bZ/2\bZ)^6.
$$

\

\noindent  \fbox{2.G24} The model is given by
$$
 X=\{w^2=x^4+y^4+z^4+\lambda x^2y^2\}\subset\bP(2,1,1,1)_{w,x,y,z}, 
$$
$$
G=\langle\sigma_{1},\sigma_2,\sigma_3\rangle=C_2^2\times C_4,\quad\sigma_1 \colon(w,x,y,z)\mapsto (-w,x,y,z),
$$
$$
\sigma_2 \colon(w,x,y,z)\mapsto (w,x,-y,z),\quad\sigma_3 \colon(w,x,y,z)\mapsto (w,x,y,\zeta_4z).
$$
The fixed curves stratification is 
$$
    \centering
    \begin{tabular}{c|c|c|c|c|c|c}
$i$&Curves $\xi_i$& $I_{\xi_i}$&$D_{\xi_i}$&$\rg(\xi_i)$&$\rg(\xi_i/D_{\xi_i})$&Standard form\\\hline
    1&$ \{w=0\}$& $C_2=\langle\sigma_1\rangle$&$G$&3&0&\multirow{4}{*}{yes}\\\cline{1-6}
    2&$ \{x=0\}$& $C_2=\langle\sigma_2\sigma_3^2\rangle$&$G$&1&0&\\\cline{1-6}
        3&$ \{y=0\}$& $C_2=\langle\sigma_2\rangle$&$G$&1&0&\\ \cline{1-6}
        4&$ \{z=0\}$& $C_4=\langle\sigma_3\rangle$&$G$&1&0&\\ 
    \end{tabular}
$$
Let $p_{ij}$ be the number of intersection points of the images of $\xi_i$ and $\xi_j$ in $X^{(1)}/G$. We record 
$$
p_{14}=2,\quad p_{12}=p_{13}=p_{23}=p_{24}=p_{34}=1.
$$
Since the intersection of any three of the four curves $\xi_i$, $i=1,2,3,4$ is empty, we have
$$
\Br([X/G])=(\bZ/2\bZ)^4.
$$

\

\noindent  \fbox{2.G44} The model is given by
$$
 X=\{w^2=x^4+y^4+z^4\}\subset\bP(2,1,1,1)_{w,x,y,z}, 
$$
$$
G=\langle\sigma_{1},\sigma_2,\sigma_3\rangle=C_2\times C_4^2,\quad\sigma_1 \colon(w,x,y,z)\mapsto (-w,x,y,z),
$$
$$
\sigma_2 \colon(w,x,y,z)\mapsto (w,x,\zeta_4y,z),\quad\sigma_3 \colon(w,x,y,z)\mapsto (w,x,y,\zeta_4z).
$$
The fixed curves stratification is 
$$
    \centering
    \begin{tabular}{c|c|c|c|c|c|c}
$i$&Curves $\xi_i$& $I_{\xi_i}$&$D_{\xi_i}$&$\rg(\xi_i)$&$\rg(\xi_i/D_{\xi_i})$&Standard form\\\hline
    1&$ \{w=0\}$& $C_2=\langle\sigma_1\rangle$&$G$&3&0&\multirow{4}{*}{yes}\\\cline{1-6}
    2&$ \{x=0\}$& $C_4=\langle\sigma_1\sigma_2\sigma_3\rangle$&$G$&1&0&\\\cline{1-6}
        3&$ \{y=0\}$& $C_4=\langle\sigma_2\rangle$&$G$&1&0&\\ \cline{1-6}
        4&$ \{z=0\}$& $C_4=\langle\sigma_3\rangle$&$G$&1&0&\\ 
    \end{tabular}
$$
Let $p_{ij}$ be the number of intersection points of the images of $\xi_i$ and $\xi_j$ in $X^{(1)}/G$. We record 
$$
p_{12}=p_{13}=p_{14}=p_{23}=p_{24}=p_{34}=1.
$$
Since the intersection of any three of the four curves $\xi_i$, $i=1,2,3,4$ is empty, we have
$$
\Br([X/G])=(\bZ/2\bZ)^2\oplus(\bZ/4\bZ).
$$

\

\noindent  \fbox{2.24.1} The model is given by
$$
 X=\{w^2=x^4+y^4+z^4+\lambda x^2y^2\}\subset\bP(2,1,1,1)_{w,x,y,z}, 
$$
$$
G=\langle\sigma_{1},\sigma_2\rangle=C_2\times C_4,\quad\sigma_1 \colon(w,x,y,z)\mapsto (w,x,-y,z),
$$
$$
\sigma_2 \colon(w,x,y,z)\mapsto (w,x,y,\zeta_4z).
$$
The fixed curves stratification is 
$$
    \centering
    \begin{tabular}{c|c|c|c|c|c|c}
$i$&Curves $\xi_i$& $I_{\xi_i}$&$D_{\xi_i}$&$\rg(\xi_i)$&$\rg(\xi_i/D_{\xi_i})$&Standard form\\\hline
    1&$ \{y=0\}$& $C_2=\langle\sigma_1\rangle$&$G$&1&0&\multirow{3}{*}{yes}\\\cline{1-6}
    2&$ \{x=0\}$& $C_2=\langle\sigma_1\sigma_2^2\rangle$&$G$&1&0&\\\cline{1-6}
        3&$ \{z=0\}$& $C_4=\langle\sigma_2\rangle$&$G$&1&0&\\ 
    \end{tabular}
$$
Let $p_{ij}$ be the number of intersection points of the images of $\xi_i$ and $\xi_j$ in $X^{(1)}/G$. We record 
$$
 p_{13}=p_{23}=2,\quad p_{12}=1.
$$
Since $\xi_1\cap\xi_2\cap\xi_3$ is empty, we have
$$
\Br([X/G])=(\bZ/2\bZ)^3.
$$

\

\noindent  \fbox{2.24.2} The model is given by
$$
 X=\{w^2=x^4+y^4+z^4+\lambda x^2y^2\}\subset\bP(2,1,1,1)_{w,x,y,z}, 
$$
$$
G=\langle\sigma_{1},\sigma_2\rangle=C_2\times C_4,\quad\sigma_1 \colon(w,x,y,z)\mapsto (-w,x,-y,z),
$$
$$
\sigma_2 \colon(w,x,y,z)\mapsto (w,x,y,\zeta_4z).
$$
The fixed curves stratification is 
$$
    \centering
    \begin{tabular}{c|c|c|c|c|c|c}
$i$&Curves $\xi_i$& $I_{\xi_i}$&$D_{\xi_i}$&$\rg(\xi_i)$&$\rg(\xi_i/D_{\xi_i})$&Standard form\\\hline
    1&$ \{z=0\}$& $C_4=\langle\sigma_2\rangle$&$G$&1&1&yes
    \end{tabular}
$$
We have
$$
\Br([X/G])=(\bZ/4\bZ)^2.
$$

\

\noindent  \fbox{2.44.1} The model is given by
$$
 X=\{w^2=x^4+y^4+z^4\}\subset\bP(2,1,1,1)_{w,x,y,z}, \quad
G=\langle\sigma_{1},\sigma_2\rangle=C_4^2,
$$
$$
\sigma_1 \colon(w,x,y,z)\mapsto (w,x,\zeta_4y,z),
\quad
\sigma_2 \colon(w,x,y,z)\mapsto (w,x,y,\zeta_4z).
$$
The fixed curves stratification is 
$$
    \centering
    \begin{tabular}{c|c|c|c|c|c|c}
$i$&Curves $\xi_i$& $I_{\xi_i}$&$D_{\xi_i}$&$\rg(\xi_i)$&$\rg(\xi_i/D_{\xi_i})$&Standard form\\\hline
    1&$ \{x=0\}$& $C_2=\langle\sigma_1^2\sigma_2^2\rangle$&$G$&1&0&\multirow{3}{*}{yes}\\\cline{1-6}
    2&$ \{y=0\}$& $C_4=\langle\sigma_1\rangle$&$G$&1&0&\\\cline{1-6}
        3&$ \{z=0\}$& $C_4=\langle\sigma_2\rangle$&$G$&1&0&\\ 
    \end{tabular}
$$
Let $p_{ij}$ be the number of intersection points of the images of $\xi_i$ and $\xi_j$ in $X^{(1)}/G$. We record 
$$
 p_{12}=p_{13}=1,\quad p_{23}=2.
$$
Since  $\xi_1\cap\xi_2\cap\xi_3$ is empty, we have
$$
\Br([X/G])=(\bZ/2\bZ)\oplus(\bZ/4\bZ).
$$

\

\noindent  \fbox{2.44.2} The model is given by
$$
 X=\{w^2=x^4+y^4+z^4\}\subset\bP(2,1,1,1)_{w,x,y,z}, \quad
G=\langle\sigma_{1},\sigma_2\rangle=C_4^2,
$$
$$
\sigma_1 \colon(w,x,y,z)\mapsto (-w,x,\zeta_4y,z),
\quad
\sigma_2 \colon(w,x,y,z)\mapsto (w,x,y,\zeta_4z).
$$
The fixed curves stratification is 
$$
    \centering
    \begin{tabular}{c|c|c|c|c|c|c}
$i$&Curves $\xi_i$& $I_{\xi_i}$&$D_{\xi_i}$&$\rg(\xi_i)$&$\rg(\xi_i/D_{\xi_i})$&Standard form\\\hline
    1&$ \{y=0\}$& $C_2=\langle\sigma_1^2\rangle$&$G$&1&0&\multirow{3}{*}{yes}\\\cline{1-6}
    2&$ \{x=0\}$& $C_4=\langle\sigma_1\sigma_2\rangle$&$G$&1&0&\\\cline{1-6}
        3&$ \{z=0\}$& $C_4=\langle\sigma_2\rangle$&$G$&1&0&\\ 
    \end{tabular}
$$
Let $p_{ij}$ be the number of intersection points of the images of $\xi_i$ and $\xi_j$ in $X^{(1)}/G$. We record 
$$
 p_{12}=p_{13}=p_{23}=1.
$$
Since  $\xi_1\cap\xi_2\cap\xi_3$ is empty, we have
$$
\Br([X/G])=\bZ/2\bZ.
$$

\noindent$\bullet$ {\bf Automorphisms of Del Pezzo surfaces of degree $1$}\\

\noindent  \fbox{1.B2.1} The model is given by
$$
X=\{w^2=z^3+zL_2(x^2,y^2)+L_3(x^2,y^2)\}\subset\bP(3,1,1,2)_{w,x,y,z},
$$
$$
G=C_2^2,\quad\sigma_1\colon (w,x,y,z)\mapsto(-w,x,y,z),
\,\,
\sigma_2\colon (w,x,y,z)\mapsto(w,x,-y,z).
$$
The fixed curves stratification is 
$$
    \centering
    \begin{tabular}{c|c|c|c|c|c|c}
$i$&Curves $\xi_i$& $I_{\xi_i}$&$D_{\xi_i}$&$\rg(\xi_i)$&$\rg(\xi_i/D_{\xi_i})$&Standard form\\\hline
    1&$ \{w=0\}$& $C_2=\langle\sigma_1\rangle$&$G$&4&1&\multirow{2}{*}{yes}\\\cline{1-6}
    2&$ \{y=0\}$& $C_2=\langle\sigma_2\rangle$&$G$&1&0&\\ 
    \end{tabular}
$$
The images of $\xi_1$ and $\xi_2$ in $X^{(1)}/G$ intersect in three points. We have
$$
\Br([X/G])=(\bZ/2\bZ)^4.
$$


\

\noindent  \fbox{1.$\sigma\rho$2.1} The model is given by
$$
X=\{w^2=z^3+L_3(x^2,y^2)\}\subset\bP(3,1,1,2)_{w,x,y,z},
\quad
G=C_6\times C_2,
$$
$$
\sigma_1\colon (w,x,y,z)\mapsto(-w,x,y,\zeta_3z),
\quad
\sigma_2\colon (w,x,y,z)\mapsto(w,x,-y,z).
$$
The fixed curves stratification is 
$$
    \centering
    \begin{tabular}{c|c|c|c|c|c|c}
$i$&Curves $\xi_i$& $I_{\xi_i}$&$D_{\xi_i}$&$\rg(\xi_i)$&$\rg(\xi_i/D_{\xi_i})$&Standard form\\\hline
    1&$ \{w=0\}$& $C_2=\langle\sigma_1^3\rangle$&$G$&4&0&\multirow{3}{*}{yes}\\\cline{1-6}
    2&$ \{y=0\}$& $C_2=\langle\sigma_2\rangle$&$G$&1&0&\\\cline{1-6}
    3&$ \{z=0\}$& $C_3=\langle\sigma_1^2\rangle$&$G$&2&0&\\ 
    \end{tabular}
$$
The images of $\xi_1$ and $\xi_2$ in $X^{(1)}/G$ intersect in one point. We have
$$
\Br([X/G])=0.
$$

\

\noindent  \fbox{1.$\sigma\rho$3} The model is given by
$$
X=\{w^2=z^3+L_2(x^3,y^3)\}\subset\bP(3,1,1,2)_{w,x,y,z},
\quad
G=C_6\times C_3,
$$
$$
\sigma_1\colon (w,x,y,z)\mapsto(-w,x,y,\zeta_3z),
\quad
\sigma_2\colon (w,x,y,z)\mapsto(w,x,\zeta_3y,z).
$$
The fixed curves stratification is 
$$
    \centering
    \begin{tabular}{c|c|c|c|c|c|c}
$i$&Curves $\xi_i$& $I_{\xi_i}$&$D_{\xi_i}$&$\rg(\xi_i)$&$\rg(\xi_i/D_{\xi_i})$&Standard form\\\hline
    1&$ \{w=0\}$& $C_2=\langle\sigma_1^3\rangle$&$G$&4&0&\multirow{4}{*}{yes}\\\cline{1-6}
    2&$ \{y=0\}$& $C_3=\langle\sigma_2\rangle$&$G$&1&0&\\\cline{1-6}
     3&$ \{x=0\}$& $C_3=\langle\sigma_1^2\sigma_2\rangle$&$G$&1&0&\\ \cline{1-6}
   4&$ \{z=0\}$& $C_3=\langle\sigma_1^2\rangle$&$G$&2&0&\\ 
    \end{tabular}
$$
Let $p_{ij}$ be the number of intersection points of the images of $\xi_i$ and $\xi_j$ in $X^{(1)}/G$. We record 
$$
 p_{23}=p_{24}=p_{34}=1.
$$
Since  $\xi_2\cap\xi_3\cap\xi_4$ is empty, we have
$$
\Br([X/G])=\bZ/3\bZ.
$$

\

\noindent  \fbox{1.$\rho$3} The model is given by
$$
X=\{w^2=z^3+L_2(x^3,y^3)\}\subset\bP(3,1,1,2)_{w,x,y,z},
\quad
G=C^2_3,
$$
$$
\sigma_1\colon (w,x,y,z)\mapsto(w,x,y,\zeta_3z),
\quad
\sigma_2\colon (w,x,y,z)\mapsto(w,x,\zeta_3y,z).
$$
The fixed curves stratification is 
$$
    \centering
    \begin{tabular}{c|c|c|c|c|c|c}
$i$&Curves $\xi_i$& $I_{\xi_i}$&$D_{\xi_i}$&$\rg(\xi_i)$&$\rg(\xi_i/D_{\xi_i})$&Standard form\\\hline
    1&$ \{z=0\}$& $C_3=\langle\sigma_1\rangle$&$G$&2&0&\multirow{3}{*}{yes}\\\cline{1-6}
    2&$ \{y=0\}$& $C_3=\langle\sigma_2\rangle$&$G$&1&0&\\\cline{1-6}
     3&$ \{x=0\}$& $C_3=\langle\sigma_1^2\sigma_2\rangle$&$G$&1&0&\\
    \end{tabular}
$$
Let $p_{ij}$ be the number of intersection points of the images of $\xi_i$ and $\xi_j$ in $X^{(1)}/G$. We record 
$$
 p_{12}=p_{13}=2,\quad p_{23}=1.
$$
Since  $\xi_1\cap\xi_2\cap\xi_3$ is empty, we have
$$
\Br([X/G])=(\bZ/3\bZ)^3.
$$

\


\noindent  \fbox{1.B4.1} The model is given by
$$
X=\{w^2=z^3+zL_1(x^4,y^4)+x^2L_1'(x^4,y^4)\}\subset\bP(3,1,1,2)_{w,x,y,z},
$$
$$
G=C_2\times C_4,
$$
$$
\sigma_1\colon (w,x,y,z)\mapsto(-w,x,y,z),
\quad
\sigma_2\colon (w,x,y,z)\mapsto(w,x,\zeta_4y,z).
$$
The fixed curves stratification is 
$$
    \centering
    \begin{tabular}{c|c|c|c|c|c|c}
$i$&Curves $\xi_i$& $I_{\xi_i}$&$D_{\xi_i}$&$\rg(\xi_i)$&$\rg(\xi_i/D_{\xi_i})$&Standard form\\\hline
    1&$ \{w=0\}$& $C_2=\langle\sigma_1\rangle$&$G$&4&0&\multirow{3}{*}{yes}\\\cline{1-6}
    2&$ \{x=0\}$& $C_2=\langle\sigma_1\sigma_2^2\rangle$&$G$&1&0&\\\cline{1-6}
     3&$ \{y=0\}$& $C_4=\langle\sigma_2\rangle$&$G$&1&0&\\
    \end{tabular}
$$
Let $p_{ij}$ be the number of intersection points of the images of $\xi_i$ and $\xi_j$ in $X^{(1)}/G$. We record 
$$
p_{12}=2,\quad p_{13}=3,\quad p_{23}=1.
$$
Since  $\xi_1\cap\xi_2\cap\xi_3$ is empty, we have
$$
\Br([X/G])=(\bZ/2\bZ)^4.
$$

\

\noindent  \fbox{1.B6.1} The model is given by
$$
X=\{w^2=z^3+\lambda zx^4+\mu x^6+y^6\}\subset\bP(3,1,1,2)_{w,x,y,z},
\quad
G=C_2\times C_6,
$$
$$
\sigma_1\colon (w,x,y,z)\mapsto(-w,x,y,z),
\quad
\sigma_2\colon (w,x,y,z)\mapsto(w,x,-\zeta_3y,z).
$$
The fixed curves stratification is 
$$
    \centering
    \begin{tabular}{c|c|c|c|c|c|c}
$i$&Curves $\xi_i$& $I_{\xi_i}$&$D_{\xi_i}$&$\rg(\xi_i)$&$\rg(\xi_i/D_{\xi_i})$&Standard form\\\hline
    1&$ \{w=0\}$& $C_2=\langle\sigma_1\rangle$&$G$&4&0&\multirow{3}{*}{yes}\\\cline{1-6}
    2&$ \{x=0\}$& $C_2=\langle\sigma_1\sigma_2^3\rangle$&$G$&1&0&\\\cline{1-6}
     3&$ \{y=0\}$& $C_6=\langle\sigma_2\rangle$&$G$&1&0&\\
    \end{tabular}
$$
Let $p_{ij}$ be the number of intersection points of the images of $\xi_i$ and $\xi_j$ in $X^{(1)}/G$. We record 
$$
p_{13}=3,\quad p_{12}=p_{23}=1.
$$
Since  $\xi_1\cap\xi_2\cap\xi_3$ is empty, we have
$$
\Br([X/G])=(\bZ/2\bZ)^3.
$$

\

\noindent  \fbox{1.$\sigma\rho$6} The model is given by
$$
X=\{w^2=z^3+x^6+y^6\}\subset\bP(3,1,1,2)_{w,x,y,z},
\quad
G=C_6^2,
$$
$$
\sigma_1\colon (w,x,y,z)\mapsto(-w,x,y,\zeta_3z),
\quad
\sigma_2\colon (w,x,y,z)\mapsto(w,x,-\zeta_3y,z).
$$
The fixed curves stratification is 
$$
    \centering
    \begin{tabular}{c|c|c|c|c|c|c}
$i$&Curves $\xi_i$& $I_{\xi_i}$&$D_{\xi_i}$&$\rg(\xi_i)$&$\rg(\xi_i/D_{\xi_i})$&Standard form\\\hline
    1&$ \{w=0\}$& $C_2=\langle\sigma_1^3\rangle$&$G$&4&0&\multirow{4}{*}{yes}\\\cline{1-6}
    2&$ \{z=0\}$& $C_3=\langle\sigma_1^2\rangle$&$G$&2&0&\\\cline{1-6}
    3&$ \{x=0\}$& $C_6=\langle\sigma_1^5\sigma_2\rangle$&$G$&1&0&\\\cline{1-6}
    4&$ \{y=0\}$& $C_6=\langle\sigma_2\rangle$&$G$&1&0&\\
    \end{tabular}
$$
Let $p_{ij}$ be the number of intersection points of the images of $\xi_i$ and $\xi_j$ in $X^{(1)}/G$. We record 
$$
p_{14}=3,\quad p_{13}=p_{23}=p_{24}=p_{34}=1.
$$
Since the intersection of any three of the four curves $\xi_i$, $i=1,2,3,4$ is empty, we have
$$
\Br([X/G])=(\bZ/2\bZ)^3\oplus(\bZ/3\bZ).
$$

\

\noindent  \fbox{1.$\rho$6} The model is given by
$$
X=\{w^2=z^3+x^6+y^6\}\subset\bP(3,1,1,2)_{w,x,y,z},
\quad
G=C_3\times C_6,
$$
$$
\sigma_1\colon (w,x,y,z)\mapsto(w,x,y,\zeta_3z),
\quad
\sigma_2\colon (w,x,y,z)\mapsto(w,x,-\zeta_3y,z).
$$
The fixed curves stratification is 
$$
    \centering
    \begin{tabular}{c|c|c|c|c|c|c}
$i$&Curves $\xi_i$& $I_{\xi_i}$&$D_{\xi_i}$&$\rg(\xi_i)$&$\rg(\xi_i/D_{\xi_i})$&Standard form\\\hline
    1&$ \{z=0\}$& $C_3=\langle\sigma_1\rangle$&$G$&2&0&\multirow{3}{*}{yes}\\\cline{1-6}
    2&$ \{x=0\}$& $C_3=\langle\sigma_1\sigma_2^2\rangle$&$G$&1&0&\\\cline{1-6}
    3&$ \{y=0\}$& $C_6=\langle\sigma_2\rangle$&$G$&1&0&\\
    \end{tabular}
$$
Let $p_{ij}$ be the number of intersection points of the images of $\xi_i$ and $\xi_j$ in $X^{(1)}/G$. We record 
$$
p_{13}=2,\quad p_{12}=p_{23}=1.
$$
Since  $\xi_1\cap\xi_2\cap\xi_3$ is empty, we have
$$
\Br([X/G])=(\bZ/3\bZ)^2.
$$

\

\noindent  \fbox{1.B6.2} The model is given by
$$
X=\{w^2=z^3+\lambda zx^2y^2+x^6+y^6\}\subset\bP(3,1,1,2)_{w,x,y,z},
\quad
G=C_2\times C_6,
$$
$$
\sigma_1\colon (w,x,y,z)\mapsto(-w,x,y,z),
\quad
\sigma_2\colon (w,x,y,z)\mapsto(w,x,-\zeta_3y,\zeta_3z).
$$
The fixed curves stratification is 
$$
    \centering
    \begin{tabular}{c|c|c|c|c|c|c}
$i$&Curves $\xi_i$& $I_{\xi_i}$&$D_{\xi_i}$&$\rg(\xi_i)$&$\rg(\xi_i/D_{\xi_i})$&Standard form\\\hline
    1&$ \{w=0\}$& $C_2=\langle\sigma_1\rangle$&$G$&4&1&\multirow{3}{*}{yes}\\\cline{1-6}
    2&$ \{x=0\}$& $C_2=\langle\sigma_1\sigma_2^3\rangle$&$G$&1&0&\\\cline{1-6}
    3&$ \{y=0\}$& $C_2=\langle\sigma_2^3\rangle$&$G$&1&0&\\
    \end{tabular}
$$
Let $p_{ij}$ be the number of intersection points of the images of $\xi_i$ and $\xi_j$ in $X^{(1)}/G$. We record 
$$
p_{12}=p_{13}=p_{23}=1.
$$
Since  $\xi_1\cap\xi_2\cap\xi_3$ is empty, we have
$$
\Br([X/G])=(\bZ/2\bZ)^3.
$$

\

\noindent  \fbox{1.B12} The model is given by
$$
X=\{w^2=z^3+\lambda zx^4+y^6\}\subset\bP(3,1,1,2)_{w,x,y,z},
\quad
G=C_2\times C_{12},
$$
$$
\sigma_1\colon (w,x,y,z)\mapsto(-w,x,y,z),
\quad
\sigma_2\colon (w,x,y,z)\mapsto(\zeta_4w,x,\zeta_{12}y,-z).
$$

The fixed curves stratification is 
$$
    \centering
    \begin{tabular}{c|c|c|c|c|c|c}
$i$&Curves $\xi_i$& $I_{\xi_i}$&$D_{\xi_i}$&$\rg(\xi_i)$&$\rg(\xi_i/D_{\xi_i})$&Standard form\\\hline
    1&$ \{w=0\}$& $C_2=\langle\sigma_1\rangle$&$G$&4&0&\multirow{3}{*}{yes}\\\cline{1-6}
    2&$ \{x=0\}$& $C_4=\langle \sigma_2^3\rangle$&$G$&1&0&\\\cline{1-6}
3&$ \{y=0\}$& $C_6=\langle \sigma_1\sigma_2^2\rangle$&$G$&1&0&\\
    \end{tabular}
$$

Let $p_{ij}$ be the number of intersection points of the images of $\xi_i$ and $\xi_j$ in $X^{(1)}/G$. We record 
$$
p_{13}=2,\quad p_{12}=p_{23}=1.
$$
Since  $\xi_1\cap\xi_2\cap\xi_3$ is empty, we have
$$
\Br([X/G])=(\bZ/2\bZ)^2.
$$


\bigskip

\subsection{Tables} 
We record the above computations in the following tables.\\

\begin{center}
\captionof*{table}{Cyclic groups $G$}
\begin{longtable}{|c|c|c|c|}
    \hline
    Label in \cite{B06} &Group $G$ & Surface $X$ & $\Br([X/G])$\\
\hline
0.n     & $C_n$&$\bP^2$&0 \\\hline
C.2&$C_2$&conic bundles&vary \\\hline
2.G&$C_2$&dP$_2$&$(\bZ/2\bZ)^{6}$ \\\hline
1.B&$C_2$&dP$_1$&$(\bZ/2\bZ)^{8}$ \\\hline
C.ro.m&$C_{2m}$&conic bundles&vary \\\hline
C.re.m&$C_{2m}$&conic bundles&vary \\\hline
3.3&$C_3$&dP$_3$&$(\bZ/3\bZ)^{2}$ \\\hline
1.$\rho$&$C_3$&dP$_1$&$(\bZ/3\bZ)^{4}$ \\\hline
2.4&$C_4$&dP$_2$&$(\bZ/4\bZ)^{2}$\\\hline
1.B2.2&$C_4$&dP$_1$&$(\bZ/2\bZ)^{4}$\\\hline
1.5&$C_5$&dP$_1$&$(\bZ/5\bZ)^{2}$\\\hline
3.6.1&$C_6$&dP$_3$&$0$\\\hline
3.6.2&$C_6$&dP$_3$&$(\bZ/2\bZ)^{2}$\\\hline
2.G3.1&$C_6$&dP$_2$&0\\\hline
2.G3.2&$C_6$&dP$_2$&$(\bZ/2\bZ)^{2}$\\\hline
2.6&$C_6$&dP$_2$&$(\bZ/3\bZ)^{2}$\\\hline
1.$\sigma\rho$&$C_6$&dP$_1$&$0$\\\hline
1.$\rho2$&$C_6$&dP$_1$&$0$\\\hline
1.B3.1&$C_6$&dP$_1$&$0$\\\hline
1.B3.2&$C_6$&dP$_1$&$(\bZ/2\bZ)^4$\\\hline
1.6&$C_6$&dP$_1$&$(\bZ/6\bZ)^2$\\\hline
1.B4.2&$C_8$&dP$_1$&$\bZ/2\bZ$\\\hline
3.9&$C_9$&dP$_3$&$0$\\\hline
1.B5&$C_{10}$&dP$_1$&$0$\\\hline
3.12&$C_{12}$&dP$_3$&$0$\\\hline
2.12&$C_{12}$&dP$_2$&$0$\\\hline
1.$\sigma\rho$2.2&$C_{12}$&dP$_1$&$0$\\\hline
2.G7&$C_{14}$&dP$_2$&$0$\\\hline
1.$\rho5$&$C_{15}$&dP$_1$&$0$\\\hline
2.G9&$C_{18}$&dP$_2$&$0$\\\hline
1.B10&$C_{20}$&dP$_1$&$0$\\\hline
1.$\sigma\rho4$&$C_{24}$&dP$_1$&$0$\\\hline
1.$\sigma\rho5$&$C_{30}$&dP$_1$&$0$\\\hline
\end{longtable}
\end{center}

\begin{center}
\captionof*{table}{Noncyclic groups $G$}
\begin{longtable}{|c|c|c|c|}
    \hline
    Label in \cite{B06} &Group $G$ & Surface $X$ & $\Br([X/G])$\\
\hline
0.mn&$C_n\times C_{m}$&$\bP^1\times\bP^1$&$\bZ/\mathrm{gcd}(m,n)\bZ$ \\\hline
P1.22n&$C_2\times C_{2n}$&$\bP^1\times\bP^1$&$0$ \\\hline
P1.222n&$C_2^2\times C_{2n}$&$\bP^1\times\bP^1$&$(\bZ/2\bZ)^2$ \\\hline
P1.22.1&$C_2^2$&$\bP^1\times\bP^1$&$0$ \\\hline
P1.222&$C_2^3$&$\bP^1\times\bP^1$&$(\bZ/2\bZ)^2$ \\\hline
P1.2222&$C_2^4$&$\bP^1\times\bP^1$&$(\bZ/2\bZ)^4$ \\\hline
P1s.24&$C_2\times C_{4}$&$\bP^1\times\bP^1$&0 \\\hline
P1s.222&$C_2^3$&$\bP^1\times\bP^1$& ${(\bZ/2\bZ)^3}$ \\\hline
0.V9&$C_3^2$&$\bP^2$& $0$ \\\hline
4.222&$C_2^3$&dP$_4$& $(\bZ/2\bZ)^4$ \\\hline
4.2222&$C_2^4$&dP$_4$& $(\bZ/2\bZ)^6$ \\\hline
4.42&$C_4\times C_2$&dP$_4$& $(\bZ/2\bZ)^2$ \\\hline
3.33.1&$C_3^2$&dP$_3$&$(\bZ/3\bZ)^2$  \\\hline
3.33.2&$C_3^2$&dP$_3$&$(\bZ/3\bZ)^2$  \\\hline
3.36&$C_3\times C_6$&dP$_3$&$\bZ/3\bZ$  \\\hline
3.333&$C_3^3$&dP$_3$&$(\bZ/3\bZ)^3$  \\\hline
2.G2&$C_2^2$&dP$_2$&$(\bZ/2\bZ)^5$ \\\hline
2.G4.1&$C_2\times C_4$&dP$_2$&$(\bZ/2\bZ)^3$ \\\hline
2.G4.2&$C_2\times C_4$&dP$_2$&$(\bZ/2\bZ)^3$ \\\hline
2.G6&$C_2\times C_6$&dP$_2$&$\bZ/2\bZ$ \\\hline
2.G8&$C_2\times C_8$&dP$_2$&$(\bZ/2\bZ)^2$ \\\hline
2.G12&$C_2\times C_{12}$&dP$_2$&$\bZ/2\bZ$ \\\hline
2.G22&$C_2^3$&dP$_2$&$(\bZ/2\bZ)^6$ \\\hline
2.G24&$C_2^2\times C_4$&dP$_2$&$(\bZ/2\bZ)^4$ \\\hline
2.G44&$C_2\times C_4^2$&dP$_2$&$(\bZ/2\bZ)^2\oplus(\bZ/4\bZ)$ \\\hline
2.24.1&$C_2\times C_4$&dP$_2$&$(\bZ/2\bZ)^3$ \\\hline
2.24.2&$C_2\times C_4$&dP$_2$&$(\bZ/4\bZ)^2$ \\\hline
2.44.1&$C_4^2$&dP$_2$&$(\bZ/2\bZ)\oplus(\bZ/4\bZ)$\\\hline
2.44.2&$C_4^2$&dP$_2$&$\bZ/2\bZ$\\\hline
1.B2.1&$C_2^2$&dP$_1$&$(\bZ/2\bZ)^4$\\\hline
1.$\sigma\rho2$.1&$C_6\times C_2$&dP$_1$&0\\\hline
1.$\sigma\rho3$&$C_6\times C_3$&dP$_1$&$\bZ/3\bZ$\\\hline
1.$\rho3$&$C_3^2$&dP$_1$&$(\bZ/3\bZ)^3$\\\hline
1.B4.1&$C_2\times C_4$&dP$_1$&$(\bZ/2\bZ)^4$\\\hline
1.B6.1&$C_2\times C_6$&dP$_1$&$(\bZ/2\bZ)^3$\\\hline
1.$\sigma\rho$6&$C_6^2$&dP$_1$&$(\bZ/2\bZ)^3\oplus(\bZ/3\bZ)$\\\hline
1.$\rho$6&$C_3\times C_6$&dP$_1$&$(\bZ/3\bZ)^2$\\\hline
1.B6.2&$C_2\times C_6$&dP$_1$&$(\bZ/2\bZ)^3$\\\hline
1.B12&$C_2\times C_{12}$&dP$_1$&$(\bZ/2\bZ)^2$\\\hline
\end{longtable}
\end{center}

\end{document}